\documentclass[12pt]{article}
\usepackage{blindtext}
\usepackage{hyperref}
\usepackage{setspace}
\usepackage{float}
\hypersetup{
	colorlinks=true,
	linkcolor=blue,
}
\usepackage[automark]{scrlayer-scrpage}
\usepackage[utf8]{inputenc}  

\usepackage{tocloft}
\usepackage{blindtext}
\usepackage{subcaption}
\usepackage[shortcuts]{extdash}

\usepackage{amsthm}
\usepackage{enumitem}
\usepackage[ruled,vlined]{algorithm2e}
\usepackage{amsmath}           
\usepackage{amsfonts}          
\usepackage{amssymb}
\usepackage{mathalfa}           
\usepackage{graphicx}          
\usepackage[T1]{fontenc}       
\usepackage{ae}                

\usepackage{lastpage}          
\usepackage[margin=10pt,font=small,labelfont=bf]{caption} 
\usepackage[T1]{fontenc}
\usepackage{enumitem}
\usepackage{stmaryrd}
\usepackage[arrow, matrix, curve]{xy}
\usepackage{listings}
\usepackage{color} 
\definecolor{mygreen}{RGB}{28,172,0} 
\definecolor{mylilas}{RGB}{170,55,241}

\usepackage{tikz}
\usetikzlibrary{knots,calc,shapes.geometric}
\usetikzlibrary{positioning}

\newcommand{\RR}{\mathbb{R}}

\newcommand{\NN}{\mathbb{N}}

\newcommand{\Nhat}{\mathcal{\hat{N}}}
\newcommand{\Nlim}{\mathcal{N}}

\newcommand{\st}{\quad\textrm{s.t.}\quad}

\newcommand*{\norm}[1]{\left\Vert#1\right\Vert}

\newtheorem{prop}{Proposition}[section]
\newtheorem{thm}[prop]{Theorem}
\newtheorem{defn}[prop]{Definition}
\newtheorem{lem}[prop]{Lemma}

\newtheorem{algo}[prop]{Algorithm}

\theoremstyle{definition}
\newtheorem{remark}[prop]{Remark}
\newtheorem{example}[prop]{Example}

\usepackage{calrsfs}
\DeclareMathAlphabet{\pazocal}{OMS}{zplm}{m}{n}

\makeatletter
\newcommand\nocaption{%
    \renewcommand\p@subfigure{}
    \renewcommand\thesubfigure{\thefigure\alph{subfigure})}
}
\makeatother

\makeatother

\usepackage{pdfpages}
\usepackage{layout}
\usepackage{comment}

\begin{document}
\author{Christian Kanzow and Felix Weiß}
\title{An Exact Penalty Approach for General $ \ell_0 $-Sparse Optimization Problems}
\maketitle

\begin{abstract}
We consider the general nonlinear optimization problem where the 
objective function has an additional term defined by the
$ \ell_0 $-quasi-norm in order to
promote sparsity of a solution. This problem is highly difficult 
due to its nonconvexity and discontinuity. We generalize some
recent work and present a whole class of reformulations of
this problem consisting of smooth nonlinear programs. This
reformulated problem is shown to be equivalent to the original
$ \ell_0 $-sparse optimization problem both in terms of local
and global minima. The reformulation contains a complementarity
constraint, and exploiting the particular structure of 
this reformulated problem, we introduce several problem-tailored
constraint qualifications, first- and second-order optimality
conditions and develop an exact penalty-type method which is shown
to work extremely well on a whole bunch of different applications. 
\end{abstract}

{
\par\addvspace{\baselineskip}
}

\section{Introduction}\label{Sec:Intro}

In this paper, we consider the sparse optimization problem of the form
\begin{equation}
	\min_x f(x) + \rho \norm{x}_0, \quad x \in X, \tag{SPO} \label{SPO}
\end{equation} 
where $f \colon \RR^n \to \RR $ is a smooth function, $ X \subseteq 
\RR^n $ a nonempty and closed set, $ \rho > 0 $ a given scalar, and
\begin{equation*}
	\norm{x}_0 := \textrm{number of nonzero components } x_i 
	\textrm{ of } x.
\end{equation*}
Note that $ \norm{\cdot}_0 $ is not a norm, though it is often referred
to as the \emph{$ \ell_0 $-norm} in the literature. We call 
\eqref{SPO} also the \emph{sparsest} optimization problem since we
really want to solve this problem with the $ \ell_0 $-norm, and do
not approximate this expression by some weaker version like in the 
standard approach, where the term $ \norm{x}_0 $ gets approximated
by $ \norm{x}_1 $ or some other (nicer) function. In the first part of
the paper, we deal with an abstract feasible set $ X $, whereas later,
in the algorithmic part, we will assume that $ X $ is described 
suitably by some equality and inequality constraints.

The solution of the sparsest optimization problem \eqref{SPO} is 
highly difficult due to the nonconvexity and discontinuity of the 
$ \ell_0 $-term in the objective function. According to 
\cite{le2015dc}, existing solution methods for \eqref{SPO} can be 
divided into the following categories: (a) convex approximations,
(b) nonconvex approximations, and (c) nonconvex exact reformulations.

The convex approximation schemes typically replace the 
$ \ell_0 $-norm by the $ \ell_1 $-norm. This is the most standard
approach which works very successfully in many applications.
Furthermore, it has the major advantage that the resulting 
optimization problem is convex provided that the objective
function $ f $ and the feasible set $ X $ are convex. The
$ \ell_1 $-norm makes this problem nondifferentiable, 
nevertheless, there are plenty of methods which can be applied
to this nonsmooth convex problem, see, e.g., the excellent
monograph \cite{Beck2017} for many examples of this kind.

The class of nonconvex approximation schemes usually replaces the 
$ \ell_0 $-term in \eqref{SPO} by a nonconvex penalty function.
One possibility is to use the $ \ell_p $-quasi-norm for 
$ p \in (0,1) $, see, e.g., \cite{ChenGuoLuYe2017,DeMarchiKanzow2023}, 
which has nicer properties than the $ \ell_0 $-norm, e.g., it
is continuous. However, despite its nonconvexity, it also
fails to be Lipschitz continuous. There exist several other 
nonconvex penalty functions with the aim to approximate the
$ l_0 $-norm in a suitable way and to keep some nicer smoothness
assumptions like SCAD (= smoothly clipped absolute deviation) \cite{FanLi2001}, MCP (= minimax concave penalty) \cite{Zhang2010}, 
PiE (= piecewise exponential) \cite{Nguyen2015} or the transformed $ \ell_1 $ approach \cite{ZhangYin2018}. The penalty 
decomposition algorithm from \cite{Lu2013} is another
approximation scheme for the solution of \eqref{SPO} and based
on the quadratic (inexact) penalty function.
Note that many of these
techniques are investigated only for particular classes of
problems covered by \eqref{SPO}.

Here, we are particularly interested in the third class, the 
exact (nonconvex) reformulations of the sparsest optimization problem.
There exist exact reformulations of \eqref{SPO} as mixed-integer
programs, see \cite{Bienstock1996} and the recent survey article 
\cite{Tillmann2021} for further references. At least for convex quadratic
programs with an additional $ \ell_0 $-term, this allows to 
compute a global minimum by suitable solvers if they get enough
time.  Another type of reformulation has been developed on the back of DC-approaches (DC = difference of convex) in \cite{le2015dc}, where DC-functions were used to approximate the $\ell_0$-norm, and in \cite{Gotoh2017}, where an exact reformulation of the $\ell_0$-norm was featured (though mainly in the context of cardinality-constrained 
problems, see below).
The paper \cite{Feng2018} considers an approach where the $ \ell_0 $-term
is replaced by a suitable complementarity constraint. 
Subsequently, the latter was shown to be equivalent to \eqref{SPO}
both in terms of local and global minima by the authors in 
\cite{KanzowSchwartzWeiss2022}.

The current work generalizes the recent contribution from
\cite{KanzowSchwartzWeiss2022} by introducing and investigating a whole class of 
reformulations of \eqref{SPO}. The main idea presented here is
somewhat related to a similar technique for 
cardinality-constrained optimization problems discussed in \cite{CervinkaKanzowSchwartz2016, BurdakovKanzowSchwartz2016}
where the $ \ell_0 $-term in the objective function is replaced 
by a constraint of the form $ \| x \|_0 \leq s $ for some given 
$ s \in \mathbb{N} $. Note, however, that it is not possible 
to reformulate cardinality-constrained problems into 
a sparsest optimization problem, see \cite{Tillmann2021} for a counterexample.

The class of reformulations presented here is
also related to optimization problems called mathematical programs with equilibrium or switching constraints (MPEC and MPSC, for short). The stationarity conditions developed here are mostly in the same vein. Globalization approaches for these types of problems include, for instance, relaxation and (exact) penalty methods, cf.\ \cite{KanzowMehlitzSteck2019, Ralph2004, Liang2021}. Due to the special (almost separable) structure
the equililibrium (complementarity) or switching constraints arise in our
class of reformulations, we are able to prove relatively strong results
which go far beyond those which are known for general MPECs or MPSCs.

The paper is organized as follows: Section~\ref{Sec:Background}
presents some background material from optimization and 
variational analysis. In Section~\ref{Sec:Reformulation}, we 
introduce our class of reformulations of problem \eqref{SPO}
and show that both the local and global minima coincide with the
global and local minima of \eqref{SPO} (note that this is in contrast
to the related reformulation of cardinality constraints discussed
in \cite{CervinkaKanzowSchwartz2016, BurdakovKanzowSchwartz2016} 
where the reformulated problem might have additional local minima).
We then introduce several problem-tailored constraint qualifications
in Section~\ref{Sec:CQs} and present the resulting first- and
second-order conditions for problem \eqref{SPO}. An approximate
stationarity concept will be discussed in Section~\ref{Sec:AS-Stationarity}. We then present our exact
penalty method in \ref{Sec:Algorithm} and provide several 
(strong) exactness and convergence results. In Section~\ref{Sec:Numerics},
we then investigate the numerical behaviour of our methods applied
to a variety of different applications, which indicates that our
method usually gets high-quality solutions, in many cases equal
to the global minimum (for those problems where the global minimum
is known or computable). We conclude with some final remarks in
Section~\ref{Sec:Final}.
 

We close with some remarks on our notation in use: In the various parts of this paper, we address via 
$$
   I_0(x):= \{ i \, | \, x_i = 0\}
$$ 
the set of indices for which $x$ vanishes. Furthermore, we write $x \circ y$ for the Hadamard-product of $x$ and $y$, i.e.\ the component wise multiplication of the two vectors. We abbreviate the canonical unit vector by $e_i \in \RR^{n}$, indicating that the single $1$ is in the $i$-th position, and additionally write $e := (1,1,...,1)^T \in \RR^{n}$. Since we will introduce sign constraints to our variables, we also denote with $\RR^n_+$ the cone of vectors with only non-negative entries in $\RR^n$.

\section{Mathematical Background}\label{Sec:Background}

This section provides some background from mathematical optimization
and variational analysis, see, e.g., the monographs \cite{Bertsekas1999,NocedalWright1999} and \cite{Mordukhovich2018,RockafellarWets2009}, respectively, for 
more details and corresponding proofs. 

Consider the optimization problem 
\begin{equation*}
	\min \ f(x) \quad \st \quad x \in X
\end{equation*}
with a continuously differentiable objective function $ f \colon \RR^n 
\to \RR $ and a nonempty, closed set $ X \subseteq \RR^n $. For a feasible
point $ x \in X $, the \emph{(Bouligand) tangent cone} or
\emph{contingent cone} of $ x $ with respect to $ X $ is defined by 
\begin{equation*}
	\mathcal{T}_X (x) := \Big\{ d \in \RR^n \, \Big| \,
	\exists \{ x^k \} \to_X
	x, \exists \{ t_k \} \downarrow 0: d = \lim_{k \to \infty}
	\frac{x^k - x}{t_k} \Big\},
\end{equation*}
where the notation $ x^k \to_X x $ indicates a sequence $ \{ x^k \} $
converging to $ x $ such that $ x^k \in X $ for all $ k \in \NN $.
Furthermore, if the feasible set $ X $ has a representation of the form
\begin{equation*}
	X = \big\{ x \in \RR^n \, \big| \, g_i(x) \leq 0 \ ( i = 1, \ldots, m ), 
	\ h_j(x) = 0 \ (j = 1, \ldots, p) \big\}
\end{equation*}
for continuously differentiable functions $ g_i, h_j: \RR^n \to \RR $,
the corresponding \emph{linearization cone} of $ x \in X $ is defined
by
\begin{equation*}
	\mathcal{L}_X (x) := \big\{ d \in \RR^n \, \big| \, 
	\nabla g_i(x)^T d \leq 0 \ ( i \in I_g(x)), \ 
	\nabla h_j(x)^T d = 0 \ (j = 1, \ldots, p)  \big\},
\end{equation*}
where 
\begin{equation*}
	I_g(x) := \big\{ i \in \{ 1, \ldots, m \} \, \big| \, g_i(x) = 0 \big\}
\end{equation*}
denotes the set of active inequality constraints at the feasible point 
$ x $. Note that the linearization cone depends on the particular
representation of $ X $, whereas the tangent cone is a purely geometric
object, independent of any representation.

Given a nonempty cone $ C \subseteq \RR^n $, we denote by
\begin{equation*}
	C^{\circ} := \big\{ v \in \RR^n \, \big| \, v^T d \leq 0 \textrm{ for all } d \in C \big\}
\end{equation*}
the \emph{polar cone} of $ C $. We then say that the \emph{Abadie
constraint qualification} (Abadie CQ or ACQ for short) holds at
$ x \in X $ if 
\begin{equation*}
	\mathcal{T}_X (x) = \mathcal{L}_X (x)
\end{equation*}
holds (the inclusion $ \mathcal{T}_X (x) \subseteq \mathcal{L}_X (x) $
is automatically true, hence the opposite inclusion is the central
requirement). Moreover, the \emph{Guignard constraint qualification}
(Guignard CQ or simply GCQ) is satisfied at $ x \in X $ if the 
corresponding polar cones coincide, i.e., if
\begin{equation*}
	\mathcal{T}_X (x)^{\circ} = \mathcal{L}_X (x)^{\circ}
\end{equation*}
holds. Note that ACQ implies GCQ, whereas the converse is not true
in general.

For a nonempty and closed set $ X \subseteq \RR^n $, we call 
\begin{equation*}
	\Nhat_X (\bar x) := \mathcal{T}_X (\bar{x})^{\circ}
\end{equation*}
the \emph{Fr\'echet normal cone} of $ \bar{x} \in X $. Furthermore,
\begin{align*}
	\Nlim_X (\bar x) & := \textrm{Limsup}_{x \to_X \bar x} \Nhat_X (x) \\
	& := \big\{ v \in \RR^n \, \big| \, \exists \{ x^k \}, 
	\exists \{ v^k \}: x^k \to_X \bar{x}, \ v^k \to v, \ 
	v^k \in \Nhat_X (x^k) \ \forall k \in \NN \}
\end{align*}
denotes the \emph{Mordukhovich normal cone} or \emph{limiting normal 
cone} of $ \bar x \in X $. For the sake of completeness, 
we set $ \Nhat_X (\bar x) :=
\Nlim_X (\bar x) := \emptyset $ for each point $ \bar{x} \not\in X $.
Note that we always have the inclusion $ \Nhat_X (\bar x)
\subseteq \Nlim_X (\bar x) $, whereas for convex sets $ X $, both normal
cones coincide and are equal to the usual normal cone from convex 
analysis, i.e., the equalities
\begin{equation*}
	\Nhat_X (\bar x) = \Nlim_X (\bar x) =
	\mathcal{N}_X^{conv} (\bar x) := \big\{ v \in \RR^n \, \big| \,
	v^T (x- \bar x ) \leq 0 \ \forall x \in X \}
\end{equation*}
hold for convex $ X $.

Next, let us write $ \overline{\RR} := \RR \cup \{ + \infty \} $
for the extended real line (excluding the value $ - \infty $).
For $ \varphi \colon \RR^n \to \overline{\RR} $ being proper, we call 
\begin{equation*}
	\textrm{epi} (\varphi) := \big\{ (x, \alpha ) \in \RR^n \times
	\RR \, \big| \, \varphi (x) \leq \alpha \big\}
\end{equation*}
the \emph{epigraph} of $ \varphi $. Based on the previously introduced
normal cones, we can define two corres\-ponding subdifferentials for 
the nonsmooth function $ \varphi $, namely the 
\emph{Fr\'echet subdifferential} of $ \bar x \in \textrm{dom} (\varphi)
:= \{ x \in \RR^n \, | \, \varphi (x) < \infty \} $, given by
\begin{equation*}
	\hat \partial \varphi (\bar x) := \big\{ s \in \RR^n \, \big| \,
	(s, -1) \in \Nhat_{\textrm{epi}(\varphi)} \big( \bar x, \varphi (\bar x)
	\big) \big\}
\end{equation*}
and the \emph{Limiting or Mordukhovich subdifferential}
\begin{equation*}
	\partial \varphi (\bar x) := \big\{ s \in \RR^n \, \big| \,
	(s, -1) \in \Nlim_{\textrm{epi}(\varphi)} \big( \bar x, \varphi (\bar x)
	\big) \big\}.
\end{equation*}
From the corresponding relation between the normal cones, we get
the inclusion $ \hat \partial \varphi (\bar x) \subseteq \partial \varphi (\bar x) $, whereas both subdifferentials coincide and are equal to
the standard subdifferential from convex analysis, i.e.,
\begin{equation*}
	\hat \partial \varphi (\bar x) = \partial \varphi (\bar x) =
	\partial^{conv} \varphi (\bar x) := \big\{ s \in \RR^n \, \big| \,
	\varphi (x) \geq \varphi (\bar x) + s^T (x - \bar x) \ \forall
	x \in \RR^n \big\}
\end{equation*}
for $ \varphi $ being a convex function.

Using these subdifferentials, we introduce the following notion.

\begin{defn}\label{Def:Stationarity}
Consider the optimization problem 
\begin{equation}\label{Eq:OptStationary}
	\min \ \psi (x) \quad \st \quad x \in X
\end{equation}
for some proper and lower semicontinuous function $ \psi \colon \RR^n \to 
\overline{\RR} $ and a nonempty, closed set $ X \subseteq \RR^n $. We then 
call $ \bar{x} \in X $ 
\begin{itemize}
	\item[(a)] an \emph{M-stationary point} (Mordukhovich stationary
	   point) of \eqref{Eq:OptStationary} if $ 0 \in \partial \psi (\bar x)
	   + \mathcal{N}_X (\bar x) $;
	\item[(b)] an \emph{S-stationary point} (strongly stationary point)
	   of \eqref{Eq:OptStationary} if $ 0 \in \hat{\partial} \psi (\bar x)
	   + \hat{\mathcal{N}}_X (\bar x) $.
\end{itemize}
\end{defn}

\noindent
Since the Fr\'echet normal cone is (in general) smaller than the 
limiting normal cone, S\-/stationarity is a stronger stationary concept than
M-stationarity.

We next restate a sum rule for the above two subdifferentials, see
\cite[Prop.\ 1.30]{Mordukhovich2018} for a proof.

\begin{thm}\label{Thm:SumRule}
Let $ \psi := f + \varphi $ with $ \varphi \colon \RR^n \to
\overline{\RR} $ proper and lower semicontinuous, $ f \colon \RR^n \to \RR $, and $ \bar x \in \textrm{dom} (\varphi) $ be given. Then the following statements hold:
\begin{itemize}
	\item[(a)] If $ f $ is differentiable in $ \bar x $, then 
	$ \hat{\partial} \psi (\bar x) = \nabla f(\bar x) + \hat{\partial}
	\varphi (\bar x) $ holds.
	\item[(b)] If $ f $ is continuously differentiable in a neighbourhood
	of $ \bar x $, then $ \partial \psi (\bar x) = \nabla f(\bar x) + \partial \varphi (\bar x) $ holds.
\end{itemize} 
\end{thm}

\noindent
Now, consider the (constrained composite) optimization problem
\begin{equation}\label{Eq:CompOpt}
	\min \ f(x) + \varphi (x) \quad \st \quad x \in X
\end{equation}
with $ f \colon \RR^n \to \RR $ continuously differentiable,
$ \varphi \colon \RR^n \to \overline{\RR} $ proper and lower 
semicontinuous, and $ X \subseteq \RR^n $
nonempty and closed. Writing $ \psi := f + \varphi $ for
the objective function, we obtain from Theorem~\ref{Thm:SumRule}
that 
\begin{equation*}
	\hat{\partial} \psi (\bar x) = \nabla f(\bar x) + \hat{\partial}
	\varphi (\bar x) \quad \textrm{and} \quad 
	\partial \psi (\bar x) = \nabla f(\bar x) + \partial \varphi (\bar x).
\end{equation*}
Consequently, using the notion from Definition~\ref{Def:Stationarity},
we see that a feasible point $ \bar x \in X $ is M-stationary for
\eqref{Eq:CompOpt} if
\begin{equation*}
	0 \in \nabla f(\bar x) + \partial \varphi (\bar x) + 
	\Nlim_X (\bar x)
\end{equation*}
holds, whereas it is S-stationary for \eqref{Eq:CompOpt} if we have
\begin{equation*}
	0 \in \nabla f(\bar x) +  \hat \partial \varphi (\bar x) + 
	\Nhat_X (\bar x).
\end{equation*}
Note that $ \varphi (x) := \rho \norm{x}_0 $ is a proper and lower
semicontinuous function, hence our sparse optimization problem \eqref{SPO}
is a special instance of the formulation \eqref{Eq:CompOpt}. The previous
M- and S-stationary conditions then require the corresponding
subdifferentials of this function. The answer is given in the following
result, which follows from \cite{Le2012,Durea2014} and lower semicontinuity of the $\ell_0$-norm.

\begin{lem}\label{Lem:Subdiff}
Consider the function $ \varphi (x) := \rho \norm{x}_0 $ for some
$ \rho > 0 $. Then 
\begin{equation*}
	\hat \partial \varphi (x) = \partial \varphi (x) =
	\big\{ s \in \RR^n \, \big| \, s_i = 0 \textrm{ for all } i 
	\text{ with } x_i \neq 0 \big\} 
\end{equation*}
for all $ x \in \RR^n $.
\end{lem}

\noindent 
Hence the limiting and Fr\'echet subdifferentials of $ \varphi (x) = \rho \norm{x}_0 $ coincide and are independent of the particular value of
the parameter $ \rho > 0 $. In order to apply the above
M- and S-stationarity conditions of problem \eqref{Eq:CompOpt}
to our setting from \eqref{SPO}, it remains to compute the 
corresponding normal cones $ \Nlim_X (\bar x) $ and 
$ \Nhat_X (\bar x) $.
This will be done by using some (problem-tailored) constraint
qualifications, see Section~\ref{Sec:CQs}.

\section{Reformulation of Sparse Optimization Problem}\label{Sec:Reformulation}

In the previous work \cite{KanzowSchwartzWeiss2022}, we established an equivalence regarding local and (up to a scaling of $\rho$) global minima 
between problem \eqref{SPO} and the following reformulation of \eqref{SPO} based on an auxiliary variable $y$:
\begin{equation}\label{Eq:ReformOld}
        \min_{x,y} f(x) + \frac{\rho}{2} y^T(y-2e) \quad \st \quad x\in X,
        \ x\circ y = 0,
\end{equation}
where $e = (1,1,...,1)^T \in \RR^{n} $. The aim of this section is 
to generalize this approach.

To this end, we introduce a penalty function $p^{\rho} \colon \RR^n \to \RR $
(usually depending on the parameter $ \rho > 0 $) given by 
\begin{equation}\label{Eq:p-rho}
   p^\rho(y) = \sum_{i=1}^n p^{\rho}_i(y_i)
\end{equation}
with each $ p^{\rho}_i \colon \RR \to \RR $ being such that it satisfies the following conditions:
\begin{itemize}
    \item[(P.1)] $ p^\rho_i$ is convex and attains a unique minimum
    (possibly depending on $ \rho $) at some point $ s_i^\rho > 0 $;
    \item[(P.2)] $p^\rho_i(0) - p^\rho_i(s_i^\rho) = \rho$;
    \item[(P.3)] $p^\rho_i$ is sufficiently smooth.
\end{itemize}
Assumption (P.1)  simply states that $ p_i^{\rho} $ is a convex function which attains its unique minimum in the open
interval $ (0, \infty) $. We denote this minimum by $ s_i^\rho > 0$.
Furthermore, we write
\begin{equation}\label{Eq:prho-Min}
	m_i^\rho := p^\rho_i(s_i^\rho) \quad \textrm{and} \quad M^\rho := \sum_{i = 1}^n m_i^\rho
\end{equation}
for the corresponding minimal function values of $ p_i^\rho $ and 
$ p^\rho $, respectively. Condition (P.2) is a scaling assumption that
can always be guaranteed by multiplication of $ p_i^{\rho} $ with
a suitable factor, whereas condition (P.3) is a smoothness condition,
with the degree of smoothness depending on the particular situation
which should be clear from the corresponding context. In particular,
for the reformulation of the sparse optimization problem \eqref{SPO}
within this section, it will be enough to have each $ p_i^\rho $
continuous (which is automatically satisfied by the convexity
assumption). The subsequent discussion of suitable constraint
qualifications and stationarity concepts requires each $ p_i^\rho $
to be continuously differentiable, whereas in the second-order theory,
$ p_i^\rho $ needs to be twice continuously differentiable.

In the following, we provide some examples of suitable functions
$ p_i^{\rho} $.

\begin{example}\label{Ex:p-rho}
The following functions $ p_i^\rho \colon \RR \to \RR $ satisfy
conditions (P.1)--(P.3):
\begin{itemize}
	\item[(a)] The function $ p_i^{\rho} (y_i) := \rho y_i 
	(y_i-2) $ is convex (in fact, uniformly convex), satisfies all smoothness
	requirements, and attains a unique minimum at $ s_i^\rho := 1 $
	(which, in this case, is independent of $ \rho $). 
	\item[(b)] The function $p_i^\rho(y_i) = \frac{1}{2}(y_i - \sqrt{2 \rho})^2$ also satifies all of the above
	requirements and can be seen as a somewhat natural choice, since we want $y_i^*$ to attain some positive value $s_i^\rho$ for $x_i^*$ to vanish. This particular choice simply penalizes the deviation in $y_i^*$ from $s_i^\rho = \sqrt{2 \rho}$, where $s_i^\rho$ was selected in accordance to (P.2).
	\item[(c)] The shifted absolute-value function $ p_i^{\rho}
	(y_i) := \rho | y_i - 1 | $ also satisfies (P.1)--(P.3), though (P.3)
	only holds with continuity, which is sufficient for the considerations
	within this section. Using a Huber-type smoothing (together with
	a suitable scaling so that (P.2) holds), we can easily construct
	a continuously differentiable version of this function satisfying
	(P.1)--(P.3).
\end{itemize}
\end{example}

\noindent
It is clear that several other examples satisfying (P.1)--(P.3) can be
constructed easily. In the following, we assume that $ p^{\rho} $ is given
by \eqref{Eq:p-rho} with each term $ p_i^{\rho} $ satisfying conditions
(P.1)--(P.3), where only continuity is required in (P.3) within
this section. We then consider the reformulation
\begin{equation}
    \tag{SPOref} \label{SPOref}
        \min_{x,y} f(x) + p^\rho(y) \quad \st \quad x\in X, \ x\circ y = 0
\end{equation}
of the sparse optimization problem \eqref{SPO}
(the acronym "SPOref" stands for "SPO-reformula\-tion"). Note that, for the 
choice of $ p_i^\rho $ as in Example~\ref{Ex:p-rho} (a), we reobtain
the previous formulation from \eqref{Eq:ReformOld} (except for the
factor $ \frac{1}{2} $ which would destroy property (P.2)).

The aim of this section is to show that problem \eqref{SPOref} is indeed a reformulation of the sparse optimization problem in the sense that it 
has the same local and global minima. For this purpose, we begin with
the following preliminary observation.

\begin{lem}\label{0normprop}
Let $ p^\rho $ be given by \eqref{Eq:p-rho} with each $ p_i^\rho $ satisfying
properties (P.1)--(P.3), and let $ M^{\rho} $ be defined by
\eqref{Eq:prho-Min}. Then the following statements hold:
\begin{itemize}
	\item[(a)] The inequality $ \rho \norm{x}_0 \le p^\rho(y) - M^\rho $
	   holds for any feasible point $ (x,y) $ of \eqref{SPOref}.
	\item[(b)] Equality $ \rho \norm{x}_0 = p^\rho(y) - M^\rho $
	  holds for a feasible point $ (x,y) $ of \eqref{SPOref} if and only if $y_i = s_i^\rho$ for all $i \in I_0(x)$.
	\item[(c)] If $(x^*,y^*)$ is a local minimum of \eqref{SPOref}, we have $y_i^* = s_i^\rho$ for all $i \in I_0(x^*)$.
\end{itemize}
\end{lem}

\begin{proof}
(a) The claim follows from
\begin{align*}
	\rho \norm{x}_0 & = \sum_{i \not\in I_0(x)} \rho 
	= \sum_{i \not\in I_0(x)} \big( p_i^\rho (0) - p_i^\rho (s_i^\rho) 
	\big)
	= \sum_{i \not\in I_0(x)} \big( p_i^\rho (y_i) - p_i^\rho (s_i^\rho ) 
	\big) \\
	& \leq \sum_{i \not\in I_0(x)} \big( p_i^\rho (y_i) - p_i^\rho(s_i^\rho) 
	\big) + \sum_{i \in I_0(x)} \big( p_i^\rho (y_i) - p_i^\rho(s_i^\rho) 
	\big) \\
	& = \sum_{i=1}^n \big( p_i^\rho (y_i) - p_i^\rho(s_i^\rho) 
	\big)
	= p^\rho (y) - M^\rho,
\end{align*}	
where the first identity results from the definition of $ \norm{x}_0 $
together with the one of the index set $ I_0(x) $, the second equation
exploits the scaling property (P.2), the third equation comes from 
the fact that we necessarily have $ y_i = 0 $ for all $ i \not\in 
I_0(x) $ due to the constraints $ x \circ y = 0 $, the inequality
takes into account that $ p_i^\rho (y_i) - p_i^\rho (s_i^\rho) \geq 0 $
due to the minimality of $ s_i^\rho $, and the remaining part is simply
the definition of $ p^\rho $ and $ M^\rho $. \medskip

\noindent 
(b) Observe that the previous chain of equations and inequalities 
holds with equation if and only if $ p_i^\rho (y_i) = p_i^\rho (s_i^\rho) $
for all $ i \in I_0(x) $. Since the minimum $ s_i^\rho $ is unique
by condition (P.1), this holds if and only if $ y_i = s_i^\rho $ for all
$ i \in I_0(x) $, hence statement (b) holds. \medskip

\noindent 
(c) This statement follows from the observation that, for any 
$ i \in I_0(x^*) $, the auxiliary variable $ y_i^* $ has to solve the 
problem
\begin{equation*}
	\min_{y_i} \ p_i^\rho (y_i)
\end{equation*}
(note that $ y_i = 0 $ is fixed for each $ i \not\in I_0(x^*) $ due
to the complementarity-type constraint $ x \circ y = 0 $, and that
the objective function is separable in each $ y_i $).
\end{proof}

\noindent
Note that Lemma~\ref{0normprop} together with the constraint 
$ x \circ y = 0 $
implies that, for $ (x^*, y^*) $ being a local minimum of \eqref{SPOref},
we necessarily have
\begin{equation}\label{Eq:y-star}
	y^*_i = \begin{cases} s_i^\rho, & \textrm{for } i \in I_0(x^*), \\ 
	0, & \text{otherwise}.
    \end{cases}
\end{equation}
In particular, $ y^* $ is uniquely defined by $ x^* $. Exploiting 
this observation, we are able to formulate an equivalence result
between the local minima of the two problems \eqref{SPO} and \eqref{SPOref}.

\begin{thm}\label{Thm:EquivLocalMinima}
A feasible $x^*$ for \eqref{SPO} is a local minimum of \eqref{SPO} if and
only if $(x^*,y^*)$ with $ y^* $ given by \eqref{Eq:y-star} is a local
minimum of \eqref{SPOref}. 
\end{thm}

\begin{proof}
The proof is very similar to the corresponding one 
in our previous work \cite{KanzowSchwartzWeiss2022}, and is presented
here for the sake of completeness and because part of the proof will be
used also in the proof of the subsequent result

Let $x^*$ be local minimum of \eqref{SPO}, and let $y^*$ be defined by
\eqref{Eq:y-star}. We then obtain
\begin{equation}\label{Eq:Chain}
	f(x^*) + p^\rho(y^*) = f(x^*) + \rho \norm{x^*}_0 + M^\rho
	\le f(x) +\rho \norm{x}_0 + M^\rho
	\le f(x) + p^\rho(y)
\end{equation}
for all feasible $(x,y)$ with $x$ sufficiently close to $x^*$, where the first equality results from  Lemma~\ref{0normprop} (b), the subsequent
inequality exploits the local minimality of $ x^* $ for the program
\eqref{SPO}, and the final inequality follows from 
Lemma~\ref{0normprop} (a).

Conversely, let $ (x^*, y^*) $ be a local minimum of \eqref{SPOref}.
Recall that $ y^* $ is then given by \eqref{Eq:y-star}, cf.\
Lemma~\ref{0normprop} (c).
Assume that $ x^* $ is not a local minimum of
\eqref{SPO}. Then there exists a sequence $ \{ x^k \} \subseteq X $ such that
$ x^k \to x^* $ and
\begin{equation}\label{Eq:ContraLocal}
	f(x^k) + \rho \norm{x^k}_0 < f(x^*) + \rho \norm{x^*}_0 \quad
	\forall k \in \mathbb{N}.
\end{equation}
Note that $ \norm{x^k}_0 \geq \norm{x^*}_0 $ holds for all $ k $ sufficiently large. Hence we either have a subsequence $ \{ x^k \}_K $ 
such that $ \norm{x^k}_0 = \norm{x^*}_0 $ for all $ k \in K $, or $ \norm{x^*}_0 + 1 \leq \norm{x^k}_0 $ is true for almost all $ k \in
\mathbb{N} $. In the former case, it follows that $ (x^k, y^*) $ is 
feasible for \eqref{SPOref}, hence we obtain from Lemma~\ref{0normprop} (b)
and the minimality of $ (x^*, y^*) $ for \eqref{SPOref} that
\begin{align*}
	f(x^k) + \rho \norm{x^k}_0 + M^\rho & = f(x^k) + \rho \norm{x^*}_0 + M^\rho \\
	& = f(x^k) + p^\rho(y^*) \\
	& \geq f(x^*) + p^\rho(y^*) \\
	& = f(x^*) + \rho \norm{x^*}_0 + M^\rho,
\end{align*}
which contradicts \eqref{Eq:ContraLocal}. Otherwise, we have 
$ \norm{x^*}_0 + 1 \leq \norm{x^k}_0 $  and, by continuity, also
$ f(x^*) \leq f(x^k) + \rho $ for all $ k \in \mathbb{N} $ 
sufficiently large, which, in turn, gives
\begin{equation*}
	f(x^k) + \rho \norm{x^k}_0 \geq f(x^k) + \rho + \rho \norm{x^*}_0 \geq 
	f(x^*) + \rho \norm{x^*}_0.
\end{equation*}
Hence, also in this situation, we have a contradiction to \eqref{Eq:ContraLocal}.
\end{proof}

\noindent
Note that the full equivalence of the set of local minima is quite
interesting, especially since a similar result does not hold for
a somewhat related reformulation of optimization problems with
cardinality constraints, see \cite{BurdakovKanzowSchwartz2016}.

The next result states the equivalence between global minima of 
the two problems \eqref{SPO} and \eqref{SPOref}. This result, however,
is less surprising than the previous one regarding local minima
(and holds, in particular, also for the previously mentioned
cardinality-constrained problems discussed in \cite{BurdakovKanzowSchwartz2016}).
\begin{thm}\label{Thm:EquivGlobalMinima}
A feasible $x^*$ for \eqref{SPO} is a global minimum of \eqref{SPO} if and
only if $(x^*,y^*)$ with $ y^* $ given by \eqref{Eq:y-star} is a global
minimum of \eqref{SPOref}.
\end{thm}

\begin{proof}
First assume that $ x^* $ is a global minimum of \eqref{SPO}. Then
the chain of inequalities from \eqref{Eq:Chain} holds for all $ (x,y) $ 
feasible for \eqref{SPOref}, showing that $ (x^*,y^*) $ is a global
minimum of \eqref{SPOref}.

Conversely, let $ (x^*,y^*) $ solve problem \eqref{SPOref} globally,
and let $ x \in X $ be an arbitrary feasible point of \eqref{SPO}.
We then define a vector $ y \in \RR^n $ similar to \eqref{Eq:y-star}
so that $ (x,y) $ is feasible for \eqref{SPOref}. Using the
optimality of $ (x^*,y^*) $ and exploiting Lemma~\ref{0normprop} (b) 
twice, we then obtain 
\begin{equation*}
   f(x^*) + \rho \norm{x^*}_0 + M^{\rho} = f(x^*) + p^\rho (y^*) \leq
   f(x) + p^\rho (y) = f(x) + \rho \norm{x}_0 + M^\rho.
\end{equation*}
Subtracting the constant $ M^\rho $ from both sides shows that 
$ x^* $ is a global minimum of problem \eqref{SPO}.
\end{proof}

\noindent
Altogether, the results from this section show that we can 
reformulate the nonsmooth and even discontinuous sparse 
optimization problem \eqref{SPO} as a continuous problem without
any loss of information regarding local or global minima. The
main difficulty of the reformulation \eqref{SPOref} is, of course,
the constraint $ x \circ y = 0 $ which is not easy to deal with, in 
particular, it violates most of the standard constraint qualifications
like the Abadie CQ and, therefore, all constraint qualifications which
are stronger than Abadie. We are going to deal with this difficulty
in our subsequent sections by introducing some problem-tailored
CQs and by considering a particular method for the solution of
problem \eqref{SPOref} which exploits the structure of
the underlying problem.

\section{Constraint Qualifications and Optimality Conditions}\label{Sec:CQs}

The aim of this section is to introduce some suitable (problem-tailored)
constraint qualifcations which are then used to obtain corresponding
optimality conditions. Note that this requires that the feasible set
$ X $ has an explicit representation by equality and/or inequality 
constraints. In view of our subsequent algorithmic approach for 
solving the reformulation \eqref{SPOref} of the sparse optimization
problem \eqref{SPO}, we assume from now on that the feasible set
is given by
\begin{equation}\label{Eq:XRep}
	X := \big\{ x \in \RR^n \, \Big| \, g(x) \leq 0, \ h(x) = 0, \ x \geq 0
	\big\}
\end{equation}
with some smooth functions $ g \colon \RR^n \to \RR^m $ and 
$ h \colon \RR^n \to \RR^p $. Hence, in addition to some standard
constraints, we assume explicitly that the inequalities contain
nonnegativity constraints, and that these nonnegativity constraints
are separated from the remaing inequalities $ g(x) \leq 0 $. This
representation turns out to be very useful in our subsequent
exact penalty approach.

In order to derive some problem-tailored constraint qualifications,
we follow a standard approach that is also used in the context of
mathematical programs with equilibrium constraints, see 
\cite{LuoPangRalph1996}, or optimization problems with switching 
constraints, see
\cite{KanzowMehlitzSteck2019,Mehlitz2020a}. To this end, let
$ x^* $ be a local minimum of problem \eqref{SPO} with $ X $ defined
by \eqref{Eq:XRep}. 
This implies that $ x^* $ is also a local minimum of the
\emph{tightend nonlinear program}
\begin{equation}
	\min_x f(x) \quad \st \quad 
	g(x) \leq 0, \ h(x)= 0, \ x_i = 0, \ i \in I_0 (x^*)\tag{TNLP($x^*$)} \label{TNLP}
\end{equation}
since, locally, the feasible set of \eqref{TNLP} is a subset of $ X $ and 
$ x^* $, by definition, is still feasible for this problem. Note,
however, that \eqref{TNLP} depends on $ x^* $ (via the index set 
$ I_0(x^*)) $ and can therefore be used only as a theoretical tool.
Here, we exploit this observation to formulate the subsequent
constraint qualifications.

\begin{defn}\label{Def:SP-CQs}
Let $ x^* \in X $ with $ X $ being defined by \eqref{Eq:XRep}.
We say that $ x^* $ satisfies
\begin{itemize}
    \item[(a)] \emph{SP-LICQ} (sparse LICQ) if the gradients
    \begin{equation*}
       \nabla g_i(x^*), \ i \in I_g(x^*), \quad \nabla h_j(x^*), \ j = 1,...,p, \quad e_i, \ i \in I_0(x^*)
    \end{equation*}
    are linearly independent.
    \item[(b)] \emph{SP-MFCQ} (sparse MFCQ) if the gradients 
    \begin{equation*}
    \nabla g_i(x^*), \ i \in I_g(x^*), \quad  \nabla h_j(x^*), \ j = 1,...,m, \quad e_i, \ i \in I_0(x^*), \end{equation*}
    are positively linearly independent.
    \item[(c)] \emph{SP-RCPLD} (sparse RCPLD) if there is a neighborhood $U$ of $x^*$ such that, for any index sets $I_1 \subseteq \{1,...,p\} , \ I_2 \subseteq I_0(x^*)$ with the gradients 
    $ \{ \nabla h_i(x^*) \}_{I_1}, \ \{ e_i\}_{i \in I_2}$
    forming a basis of the subspace generated by all gradients 
    $ \nabla h_i(x^*) \ (i = 1, \ldots, p), \ e_i \ (i \in I_0(x^*)) $,
    the following holds for all $ x \in U $.
    \begin{enumerate}
        \item $\left( \{\nabla h_i(x) \}_{i=1}^p, \ \{e_i\}_{i \in I_0(x^*)} \right)$ is of constant rank for all $x \in U$
        \item For every $J \subseteq I_g(x^*)$, if $\left(\{\nabla h_i(x^*)\}_{I_1}, \{e_i\}_{I_2}, \{\nabla g_i(x^*)\}_{J}\right)$ is positive-linearly dependent, then $\left(\{\nabla h_i(x)\}_{I_1}, \{e_i\}_{I_2}, \{\nabla g_i(x)\}_{J}\right)$ is linearly dependent for every $x \in U$.
    \end{enumerate}
\end{itemize}
\end{defn}

\noindent 
Observe that SP-LICQ, SP-MFCQ and SP-RCPLD correspond to standard LICQ,
standard MFCQ and standard RCPLD for the tightened problem \eqref{TNLP}
(with RCPLD being a constraint qualification introduced in
\cite{AHSS-12}). Hence, 
it is clear how to formulate further constraint qualifications based
on this relation. In particular, standard theory on constraint
qualifications therefore show that each of the problem-tailored
constraint qualifications from Definition~\ref{Def:SP-CQs} imply
that (standard) Abadie CQ holds for the tightened problem
\eqref{TNLP}. We now use this observation in order to derive a
suitable optimality condition for the sparse optimization problem
\eqref{SPO}.

To this end, let $ x^* $ be a local minimum of \eqref{SPO},
and consider the corresponding tightened problem \eqref{TNLP}.
Assume further that any of the problem-tailored constraint qualifications
from Definition~\ref{Def:SP-CQs} hold. Let us denote by
\begin{equation*}
   X_{\textrm{TNLP}(x^*)} \textrm{ the feasible set of } \eqref{TNLP},
\end{equation*}
and by 
\begin{align*}
	\mathcal{T}_{X_{\textrm{TNLP}(x^*)}} (x^*), \quad 
	\mathcal{L}_{X_{\textrm{TNLP}(x^*)}} (x^*), \quad 
	\Nhat_{X_{\textrm{TNLP}(x^*)}} (x^*), \quad \textrm{and} \quad
	\Nlim_{X_{\textrm{TNLP}(x^*)}}(x^*)
\end{align*}
the corresponding tangent cone, linearization cone, Fr\'echet normal cone, and limiting normal cone of $ x^* \in X_{\textrm{TNLP}(x^*)} $.
Since each of the SP-CQs from Definition~\ref{Def:SP-CQs} implies that
the standard Abadie and, hence, the standard Guignard CQ holds 
at $ x^* \in X_{\textrm{TNLP}(x^*)} $, we obtain
\begin{equation*}
	\mathcal{T}_{X_{\textrm{TNLP}(x^*)}} (x^*)^{\circ} =
	\mathcal{L}_{X_{\textrm{TNLP}(x^*)}} (x^*)^{\circ}.
\end{equation*}
Now, the linearization cone $ \mathcal{L}_{X_{\textrm{TNLP}(x^*)}} (x^*) $ is a polyhedral convex cone, and standard results
from convex analysis imply that its polar is given by
\begin{equation*}
	\mathcal{L}_{X_{\textrm{TNLP}(x^*)}} (x^*)^{\circ} = 
	\Big\{ d \, \Big| \, d = \sum_{i \in I_g(x^*)} \lambda_i \nabla g_i(x^*) + \sum_{j=1}^p \mu_j \nabla h_j(x^*) + 
	\sum_{i \in I_0 (x^*)} \gamma_i e_i, \ \lambda_i \geq 0 \
	(i \in I_0(x^*))\Big\}	.
\end{equation*}
Since, by definition, we have 
\begin{equation*}
   \Nhat_{X_{\textrm{TNLP}(x^*)}} (x^*) = \mathcal{T}_{X_{\textrm{TNLP}(x^*)}} (x^*)^{\circ} =
   \mathcal{L}_{X_{\textrm{TNLP}(x^*)}} (x^*)^{\circ},
\end{equation*}
this shows that
\begin{equation*}
	\Nhat_{X_{\textrm{TNLP}(x^*)}} (x^*) =
	\Big\{ d \, \Big| \, d = \sum_{i \in I_g(x^*)} \lambda_i \nabla g_i(x^*) + \sum_{j=1}^p \mu_j \nabla h_j(x^*) + 
	\sum_{i \in I_0 (x^*)} \gamma_i e_i, \ \lambda_i \geq 0 \
	(i \in I_0(x^*))\Big\}.
\end{equation*}
Now, using the expression of the Fr\'echet normal cone for the
tightened nonlinear program \eqref{TNLP}, the notation of an 
S-stationary point from Section~\ref{Sec:Background} for a general
optimization problem as in \eqref{Eq:CompOpt}, and taking into account the
formula for the Fr\'echet subdifferential of the function
$ \varphi (x) := \rho \norm{x}_0 $ from Lemma~\ref{Lem:Subdiff},
it follows that the local minimum $ x^* $ satisfies the following
S-stationary conditions under any of the SP-CQs from Definition~\ref{Def:SP-CQs}, where, for simplicity of notation,
\begin{equation}\label{Eq:SP-Lagrangian}
	L^{SP} (x, \lambda, \mu) := f(x) + \lambda^T g(x) + \mu^T h(x)
\end{equation}
denotes a mapping that we call the \emph{SP-Lagrangian} of
problem \eqref{SPO} with the feasible set $ X $ given by
\eqref{Eq:XRep} (note that this SP-Lagrangian neither includes
the $ \ell_0 $-term of the original objective function nor any
term resulting from the nonnegativity constraints).

\begin{defn}\label{Def:S-stationary}
Let $ x^* \in X $ be feasible for the sparse optimization problem 
\eqref{SPO}, where $ X $ is given by \eqref{Eq:XRep}. Then we call
$ x^* $ an \emph{S-stationary point} of \eqref{SPO} if there exist
multipliers $ \lambda^* \in \RR^m $ and $ \mu^* \in \RR^p $ such that 
\begin{align*}
	\nabla_{x_i} L^{SP} (x^*, \lambda^*, \mu^*) & = 0, \quad i \not\in I_0(x^*), \\
	h(x^*) & = 0, \\
	\lambda^* \geq 0, \ g(x^*) \leq 0, \ \lambda^* \circ g(x^*) & = 0.
\end{align*}
\end{defn}

\noindent 
Note that the previous derivation shows that S-stationarity is a 
necessary optimality condition for a local minimum $ x^* $ of 
problem \eqref{SPO} provided that a suitable (problem-tailored) constraint
qualification holds. In a similar way, one can also derive an
M-stationary condition. However, in this particular case, there is
no difference between M- and S-stationarity due to the fact that
the Fr\'echet and limiting subdifferentials of the function
$ \varphi (x) = \rho \norm{x}_0 $ coincide, cf.\ Lemma~\ref{Lem:Subdiff}.

Before providing another interpretation of S-stationary points,
we consider a slightly different reformulation of problem \eqref{SPO}
with $ X $ given by \eqref{Eq:XRep}. Since, by \eqref{Eq:y-star},
any local minimum $ (x^*,y^*) $ of the reformulated problem 
\eqref{SPOref} automatically satisfies $ y^* \geq 0 $, it follows that
\eqref{SPOref} and, hence, \eqref{SPO} itself is totally equivalent
to the program
\begin{equation}
	\tag{SPOcp} \label{SPOcp}
	\min_{x,y} f(x) + p^{\rho}(y) \quad \st \quad g(x) \leq 0, \ h(x) = 0, \ x\circ y = 0, \
	x \geq 0, \ y \geq 0
\end{equation}
in terms of local and global minima. We call \eqref{SPOcp}
the complementarity reformulation of problem \eqref{SPO} (hence
the acronym "cp") due to the complementarity constraints
$ x \geq 0, y \geq 0, x^T y = 0 $. Note that \eqref{SPOcp} will
be the basis of our algorithmic approach for the solution of
the sparse optimization problem \eqref{SPO}.

Assuming that the function $ p^{\rho} $ is continuously differentiable,
the two reformulations \eqref{SPOref} and \eqref{SPOcp} are smooth
optimization problems. Hence, we can write down the corresponding
KKT conditions. It turns out they are equivalent to the
S-stationarity conditions from Definition~\ref{Def:S-stationary}.
This is summarized in the following result.

\begin{thm}\label{Thm:KKT-Equivalence}
Consider the sparse optimization problem \eqref{SPO} with feasible set
$ X $ given by \eqref{Eq:XRep}. Furthermore, for a feasible point
$ x^* \in X $, let $ y^* $ denote the corresponding vector defined by
\eqref{Eq:y-star}. Then the following statements are equivalent:
\begin{itemize}
	\item[(a)] $ x^* $ is an S-stationary point of \eqref{SPO}.
	\item[(b)] The KKT conditions of \eqref{SPOref} are satisfied
	   at $ (x^*,y^*) $.
	\item[(c)] The KKT conditions of \eqref{SPOcp} are satisfied
	   at $ (x^*,y^*) $.
\end{itemize}
\end{thm}

\noindent 
We skip the proof of Theorem~\ref{Thm:KKT-Equivalence} since it is
rather elementary. We only stress the following two facts: Definition
\eqref{Eq:y-star} of $ y^* $ implies that the bi-active set of the 
solution pair $ (x^*, y^*) $ is empty, i.e., there is no index $ i $
with $ (x_i^*, y_i^*) = (0,0) $. Furthermore, for $ x_i^* = 0 $, 
we have that $ y_i^* $ is equal to the unique minimum $ s_i^\rho $
of the function $ p_i^{\rho} $, hence $ [ \nabla p^{\rho} (y^*)]_i = 0 $
follows.

The following observation is simple, but quite interesting
and important for our subsequent theory, hence we state it 
explicitly in the following remark.

\begin{remark}\label{Rem:Biactive-empty}
Let $ (x^*, y^*) $ be a stationary point of \eqref{SPOcp}.
We then claim that the bi-active set 
$$ 
   \mathcal{B} (x^*,y^*) := \{ i \, | \ (x_i^*,y_i^*)
   = (0,0) \} 
$$ 
is automatically empty.
In fact, from the stationarity
conditions of \eqref{SPOcp}, we, in particular, obtain 
\begin{equation*}
	\nabla p_i^{\rho} (y_i^*) + \eta x_i^* - \nu_i^x = 0 
	\quad \forall i = 1, \ldots, m
\end{equation*}
for some corresponding multipliers $ \eta \in \mathbb{R} $ and 
$ \nu^x \geq 0 $. If there were an index $ i $ with $ (x_i^*,y_i^*) =
(0,0) $, the term in the middle vanishes, and the first term 
is strictly negative by the convexity of $ p_i^{\rho} $ together
with the assumption that this function attains a unique minimum
at the positive number $ s_i^{\rho} $. Hence, it follows that 
$ \nabla p_i^{\rho} (y_i^*) + \eta x_i^* - \nu_i^x < 0 $, and
this contradiction completes the proof.
\end{remark}

\noindent
Theorem~\ref{Thm:KKT-Equivalence} 
shows that the given sparse optimization problem 
\eqref{SPO} and
its two reformulations \eqref{SPOref} and \eqref{SPOcp} are not only
equivalent with respect to local and global minima, but also in terms
of their first-order optimality conditions.

We next demonstrate that also the corresponding second-order conditions
coincide. To this end, we assume that all functions are twice continuously
differentiable. Furthermore, let $ L^{SP} $ be the SP-Lagrangian of
\eqref{SPO}. We then define the \emph{SP-critical cone}
\begin{align*}
	\mathcal{C}^{SP}(x^*,\lambda^*)= 
	\big\{ d \, \big|  \, \nabla h_j(x^*)^T d &= 0,
	\ j = 1, \ldots, p, \\ \ \nabla g_i(x^*)^T d &= 0, \ i \in I_g(x^*), \ \lambda_i^* > 0, \\ \ \nabla g_i(x^*)^T d &\le 0, \ i \in I_g(x^*), \ \lambda_i^* = 0, \\ \
    d_i &= 0, \ i \in I_0(x^*) \big\}
\end{align*}
at some S-stationary point $ x^* $ with corresponding multipliers
$ (\lambda^*, \mu^*) $. 

\begin{defn} 
Given an S-stationary point $ x^* $ with multipliers
$ (\lambda^*, \mu^*) $, and using the notion of the 
SP-Lagrangian and the SP-critical cone as before, we say that the triple
$ (x^*, \lambda^*, \mu^*) $ satisfies the
\begin{itemize}
	\item[(a)] \emph{SP-SOSC} (sparse second-order sufficiency condition) if
	\begin{equation*}
	   d^T \nabla_{xx}^2 L^{SP} (x^*, \lambda^*, \mu^*) d > 0 \quad \forall d \in 
	   \mathcal{C}^{SP}(x^*,\lambda^*) \setminus \{ 0 \}
	\end{equation*}
    \item[(b)] \emph{SP-SONC} (sparse second-order necessary condition) if
    \begin{equation*}
    	d^T \nabla_{xx}^2 L^{SP} (x^*, \lambda^*, \mu^*) d \geq 0 \quad \forall d \in 
    	\mathcal{C}^{SP}(x^*,\lambda^*)
    \end{equation*}
\end{itemize}
holds.
\end{defn}

\noindent
These second-order conditions turn out to be equivalent to
the standard second-order conditions of the two reformulations
\eqref{SPOref} and \eqref{SPOcp}. This observation is summarized
in the following result.

\begin{thm} \label{Thm:SecOrdEquiv}
Let $ x^* $ be an S-stationary point with multipliers
$ (\lambda^*, \mu^*) $, and let $y^*$ be given by \eqref{Eq:y-star}. 
Assume that $p^\rho$ is twice continuously differentiable. 
Then the following statements are equivalent.
\begin{itemize}
	\item[(a)] SP-SONC holds at $x^*$ for \eqref{SPO}.
    \item[(b)] SONC holds at $(x^*,y^*)$ for \eqref{SPOref}.
    \item[(c)] SONC holds at $(x^*,y^*)$ for \eqref{SPOcp}.
\end{itemize}
The same equivalences also hold for SP\-/SOSC and SOSC provided that the 
(diagonal) matrix $ \nabla p^{\rho} (y^*) $ is positive definite.
\end{thm} 

\noindent 
Note that, due to the separable structure of the function
$ p^{\rho} $, the Jacobian $ \nabla p^{\rho} (y^*) $ is indeed
a diagonal matrix. Due to the assumed convexity of each function
$ p_i^{\rho} $, this diagonal matrix is automatically positive 
semidefinite. For the equivalence of the second-order sufficiency
conditions, however, we require the slightly stronger assumption
that this matrix is positive definite. Note that this holds 
automatically if $ p_i^{\rho} $ is strongly convex around
$ y^* $, an assumption that is satisfied for all instances from
Example~\ref{Ex:p-rho}.

We skip the details of the proof, but provide at least some hints.
The equivalence of statements (b) and (c) is based on the following
two facts: (i) the Hessian of the corresponding two Lagrangians
coincide since the Lagrangian of \eqref{SPOcp} contains
only one more linear term which disappears in the second-order 
derivatives, (ii) the criticial cones of the two sets 
$ \{ (x,y) \, | \, x \geq 0, x \circ y = 0 \} $ and 
$ \{ (x,y) \, | \, x \geq 0, y \geq 0, x \circ y = 0 \} $ coincide
at $ (x,y) = (x^*, y^*) $ since the biactive set is empty,
cf.\ Remark~\ref{Rem:Biactive-empty} (otherwise,
this statement would be wrong).

Furthermore, a simple calculation shows that the Hessian of 
the Lagrangian of either \eqref{SPOref} or \eqref{SPOcp} is given by
$$
   H = \begin{pmatrix}
	\nabla_{xx}^2 L^{SP} (x^*, \lambda^*, \mu^*) & \text{diag}(\gamma) \\
	\text{diag}(\gamma) & \nabla_{yy}^2 p^\rho(y^*)
\end{pmatrix}
$$
for some multiplier $ \gamma $. This implies
$$ 
   \begin{pmatrix}
	d_x \\ d_y 
   \end{pmatrix}^T H \begin{pmatrix}
	d_x \\ d_y
   \end{pmatrix} = 
   d_x^T \nabla_{xx}^2L^{SP}(x^*,\lambda^*, \mu^*) d_x + 2
   d_y^T \nabla_{yy}^2 p^\rho(y^*) d_y + \sum_{i=1}^n \gamma_i \left(d_x\right)_i \left(d_y\right)_i.
$$
Observe that $ d_y^T \nabla_{yy}^2 p^\rho(y) d_y \geq 0 $ holds
due to the convexity of the (separable) function $ p^\rho $, and
that $ \left(d_x\right)_i \left(d_y\right)_i = 0 $ for all $ i $
and all vectors $ (d_x,d_y) $ from the critical cone of either 
\eqref{SPOref} or \eqref{SPOcp}. Taking these observations into
account, the equivalence of statement (a) to assertions (b) or (c)
is easy to verify.

Note that, in general, we prefer to deal with
the above two sparse second-order conditions since they are defined
directly in terms of the sparse optimization problem \eqref{SPO}, 
hence they are independent of the auxiliary variable $ y $ and the function
$ p^\rho $ introduced in order to obtain the desired reformulations.

\section{Approximate S-Stationarity}\label{Sec:AS-Stationarity}

This section considers a sequential optimality condition which is the
counterpart of the \emph{approximate KKT conditions} (AKKT conditions
for short) originally introduced for standard nonlinear programs of 
the form 
\begin{equation}\label{Eq:NLP}
	\min_x \ f(x) \quad \st \quad g(x) \leq 0, \ h(x)=0
\end{equation}
by \cite{Andreani2016} (with the name \emph{cone continuity property}), 
see also \cite{BirginMartinez2014}: A feasible point $ x^* $
of \eqref{Eq:NLP} is called an \emph{AKKT point} if there are sequences 
$ \{ x^k \} \subset \RR^n $, $ \lambda^k \subset \RR^m_+ $, and $ \mu^k \subset \RR^p $ such that
\begin{equation*}
	x^k \to x^*, \quad \nabla_x L(x^k, \lambda^k, \mu^k) \to 0, \quad
	\min \big\{ \hspace*{-1.1mm} - g_i (x^k), \lambda_i^k \big\} \to 0 \ \ \forall i = 1, 
	\ldots, m,
\end{equation*}
where $ L $ denotes the (ordinary) Lagrangian of \eqref{Eq:NLP}.

The notion of an AKKT point has been generalized in different ways
to optimization problems having a special and/or difficult structure,
often coined \emph{approximate M-stationarity} (AM-stationarity)
since it is based on a sequential version of the M-stationary
optimality conditions, see, e.g., \cite{Mehlitz2020} for a 
corresponding discussion in a very general setting.

In the following, we introduce a sequential optimality condition for
the sparse optimization problem \eqref{SPO} which is based on the 
notion of an S-stationary point from Definition~\ref{Def:S-stationary}
and, therefore, takes into account the particular structure of this
problem.

\begin{defn}\label{Def:AS-stationarity}
Consider the sparse optimization problem \eqref{SPO} with feasible
set $X$ being defined by \eqref{Eq:XRep}. We then call a feasible
point $x^* \in X $ an \emph{approximate S-stationary point} 
(AS-stationary point, for
short) if there exist sequences $\{x^k\} \subset \RR^n $, $ \{ \lambda^k \} \subset \RR^m_+ $, and $ \{\mu^k\} \subset \RR^p$ such that
\begin{equation*}
   x^k \to x^*, \quad
   \nabla_{x_i} L^{SP}(x^k, \lambda^k, \mu^k) \to 0 \ \ \forall i \notin I_0(x^*), \quad
   \min\{ \hspace{-0.1mm} -g_i(x^k), \lambda_i^k\} \to 0 \ \ \forall i = 1,...,m,
\end{equation*}
where $ L^{SP} $ denotes, once again, the SP-Lagrangian from 
\eqref{Eq:SP-Lagrangian}.
\end{defn}

\noindent 
Similar to existing results on AKKT- and AM-stationary points, we can
also derive several useful properties for our notion of an AS-stationary
point in the context of sparse optimization problems. 
The first result in this context shows that any local minimum 
is automatically an AS-stationary point. Note that this statement
holds without assuming any constraint qualification.

\begin{thm}\label{Thm:MinIsAS-Stationary}
Let $x^*$ be a local minimum of the sparse optimization problem 
\eqref{SPO} with $ X $ given by \eqref{Eq:XRep}. Then $x^*$ is an 
AS-stationary point of \eqref{SPO}.
\end{thm}

\begin{proof}
First recall that the local minimum $x^*$ of \eqref{SPO} is also a local
minimum of the corresponding tightened nonlinear program
from \eqref{TNLP}. Therefore, standard results on AKKT points imply
that there exists sequences $x^k \to x^*$ as well as $\{\lambda^k\} \subset \RR^m_+$, $\{\mu^k\} \subset \RR^p$, and $\{\gamma^k\} \subset \RR^{|I_0(x^*)|}$ satisfiying
$$ 
   \nabla_x L^{SP}(x^k,\lambda^k, \mu^k ) + \sum_{i \in I_0(x^*)} \gamma_i^k e_i \to 0 \quad \text{and} \quad \min \{-g_i(x^k), \lambda_i^k\} \to 0
   \ \ \forall i = 1, \ldots, m.
$$
Hence, for all $i \notin I_0(x^*)$, this implies 
$ \nabla_{x_i} L^{SP}(x^k,\lambda^k,\mu^k) \to 0$, and the claim follows. 
\end{proof}

\noindent
Theorem~\ref{Thm:MinIsAS-Stationary} shows that AS-stationarity is
a necessary optimality condition. Of course, this does not automatically
imply that AS-stationarity is a suitable (strong) optimality condition.
In fact, it is known that, for certain classes of optimization problems
like cardinality-constrained problems, the standard AKKT-conditions
hold at every feasible point, cf.\ \cite{Ribeiro2022}. The
following example shows that this unfortunate situation does not hold
in our setting with AS-stationarity.

\begin{example}\label{Eq:AS-Stationarity}
Consider the (sparse) optimization problem
\begin{equation*} 
   \min_x \ \sum_{i = 1}^n x_i + \rho \norm{x}_0 \quad
   \st \quad \norm{x}_2^2 \le 1, \ x\ge 0.
\end{equation*}
The origin $ x = 0 $ is the only local and global minimum of this problem
and, hence, also an AS-stationary point in view of Theorem~\ref{Thm:MinIsAS-Stationary}. We claim that it is also the 
only AS-stationary point. In fact, suppose there exists an 
AS-stationary point $ x^* $ with $ x_i^* > 0 $ for at least one
component $ i $. Since $ x^* $ is an AS-stationary point, there
exist suitable sequences $ \{ x^k \} \to x^* $ and $ \{ \lambda^k \} \in 
[0, \infty ) $ such that, in particular, for the component $ i $, we have 
$ 1 + 2 \lambda^k x_i^k \to 0 $, which is impossible for all $ k $
sufficiently large since $\lambda^k \ge 0$ and $x_i^k \to x_i^*>0$. Hence, $x = 0 $ is the only AS-stationary point. Notice that the origin is also S-stationary in this example.
\end{example}

\noindent
We next want to provide conditions under which an AS-stationary point
is already S-stationary. The following example shows that this implication
does not hold in general.

\begin{example}\label{Ex:Counter-S-AS}
Consider the (sparse) optimization problem
\begin{equation*} 
	\min_x \ x_1 + \rho \norm{x}_0, \quad \st \quad 
	\frac{1}{2} \norm{x - e}_2^2 = 0, \ x \ge 0.
\end{equation*}
The only feasible point and therefore local and global minimum is $x^* = e$.
Theorem~\ref{Thm:MinIsAS-Stationary} implies that $ x^* $ is AS-stationary.
This can also be verified directly using Definition~\ref{Def:AS-stationarity}
and the sequences $x^k = (1 - \frac{1}{k}, 1,...,1)$, 
$\mu^k = k$, for which we obtain the desired limit
\begin{equation*}
   \begin{pmatrix}
   1 \\ 0 \\ \vdots \\ 0 
   \end{pmatrix} + k \cdot \begin{pmatrix}
	1 - \frac{1}{k} - 1 \\ 0 \\ \vdots \\ 0
\end{pmatrix} = 0
\end{equation*}
(note that the sequence $ \{ \mu^k \} $ is unbounded).
We claim, however, that $ x^* $ is not an S-stationary point.  
This is clear since otherwise we would have 
\begin{equation*}
   0 = \nabla_{x_1} L^{SP} (x^*, \mu^*) = 1 + \mu (x_1^* - 1) = 1
\end{equation*}
for some multiplier $ \mu \in \RR $, which is impossible.
\end{example}

\noindent
For an AS-stationary point to be S-stationary, we therefore require
a suitable constraint qualification. The following is the natural
counterpart of what is usually called AKKT-regularity or AM-regularity
in the context of standard nonlinear programs or certain structured
optimization problems, see \cite{AndreaniMartinezRamosSilva2018,Andreani2016,
BirginMartinez2014}
as well as \cite{Mehlitz2020} and references therein.

\begin{defn}
Consider the sparse optimization problem \eqref{SPO} with feasible
set $X$ being defined by \eqref{Eq:XRep}. Furthermore, let 
$ x^* \in X $ be any given feasible point. 
We say that $x^*$ satisfies the \emph{AS-regularity condition} if the cone \begin{equation*}
    K(x):= \Big\{d \, \Big| \,  d_j = \Big[ \sum_{i = 1}^p \mu_i \nabla h_i(x) + \sum_{i \in I_g(x^*)} \lambda_i \nabla g_i(x) \Big]_j , \ \mu_i \in \RR, \ \lambda_i\ge 0, \ j \notin I_0(x^*)\Big\}
\end{equation*} 
is outer semicontinuous at $x^*$, i.e.,
$
	\textrm{Limsup}_{x \to x^*} K(x) \subseteq K(x^*),
$
where 
\begin{equation*}
	\text{Limsup}_{x \to x^*} K(x) := \big\{ d \, \big| \,
	\exists \{ x^k \} \to x^*, \exists d^k \to d: d^k \in K(x^k)
	\ \forall k \in \NN \big\}
\end{equation*}
denotes the upper or outer limit of the set-valued mapping $ K( \cdot ) $.
\end{defn}

\noindent 
Note that, in the previous definition of $ K(x) $, the index set
$ I_0 (x^*) $ is fixed at $ x^* $. Moreover, we note that no condition
is required for the components $ d_j $ with $ j \in I_0(x^*) $.

The following result shows that, in a certain sense, AS-regularity is a
necessary and sufficient condition for an AS-stationary point to
be S-stationary. In particular, this means that AS-regularity is
a constraint qualification.

\begin{thm}\label{Thm:Char-S-Stat}
Let $x^*$ be feasible for \eqref{SPO} with feasible set $X$ defined 
in \eqref{Eq:XRep}. Then the following statements hold:
\begin{enumerate}
    \item If $x^*$ is an AS-stationary point 
    of \eqref{SPO} satisfying AS-regularity, then $x^*$ is an S-stationary point of \eqref{SPO}.
    \item Conversely, if for every continuously differentiable 
    objective function $f$, the implication
    \begin{center}
        $x^*$ is an AS-stationary point of \eqref{SPO} $\Longrightarrow$ $x^*$ is an S-stationary point of \eqref{SPO}
    \end{center}
    holds, then $x^*$ is AS-regular.
\end{enumerate}
\end{thm}

\begin{proof}
(a): Since $ x^* $ is an AS-stationary point, there exist sequences
$ \{ x^k \} \subset \RR^n, \{ \lambda^k \} \subset \RR^m $, and 
$ \{ \mu^k \} \subset \RR^p $ such that $ x^k \to x^* $,
$ \nabla_{x_i} L^{SP} (x^k, \lambda^k, \mu^k) \to 0 $ for all
$ i \not\in I_0(x^*) $, and $ \min \{ - g_i(x^k), \lambda_i^k \} \to 0 $
for all $ i = 1, \ldots, m $. The latter condition implies that
we may assume without loss of generality that $ \lambda_i^k = 0 $
for all $ i \not\in I_g(x^*) $ and all $ k \in \mathbb{N} $, cf.\
\cite[Thm.\ 3.2]{BirginMartinez2014} for a formal proof. Then,
writing 
\begin{equation*}
	w^k := \sum_{i \in I_g(x^*)} \lambda_i^k g_i(x^k) + \sum_{j=1}^p 
	\mu_j^k \nabla h_j(x^k)
\end{equation*}
we see that $ w^k \in K(x^k) $ for all $ k \in \mathbb{N} $, and that
$ \xi_i^k := [ \nabla f(x^k) + w^k]_i \to 0 $ for all $ i \not\in 
I_0(x^*) $. Define $ \xi_i^k := 0 $ for the remaining components
$ i \in I_0(x^k) $ and set $ v^k := \xi^k - \nabla f(x^k) $. We then 
obtain $ v^k \in K(x^k) $ for all $ k \in \mathbb{N} $ since, for
the relevant components $ i \not\in I_0(x^*) $, we have 
$ v_i^k = \xi_i^k - [\nabla f(x^k)]_i = w_i^k $. The assumed 
AS-regularity of $ x^* $ then implies 
$$   
   -\nabla f(x^*) = \lim_{k \to \infty} v^k \in \textrm{Limsup}_{x \to x^*} K(x) \subseteq K(x^*),
$$   
which shows that $x^*$ is S-stationary.\medskip

\noindent
(b): Take $w^* \in \textrm{Limsup}_{x \to x^*} K(x)$ arbitrarily. 
Then there is $\{(x^k,w^k)\} \to (x^*,w^*)$ with $w^k \in K(x^k)$ for all $k\in \NN$. Hence,
there exist sequences $ \{ \lambda^k \} \subset 
\RR^{|I_g(x^*)|}_+ $ and $ \{ \mu^k \} \subset \RR^p $ such that
\begin{equation*} 
	w^k_i =\sum_{j \in I_g(x^*)} \lambda_j^k \nabla g_j(x^k)_i
	+ \sum_{j = 1}^p \mu_j^k \nabla h_j(x^k)_i \quad \forall i \not\in 
	I_0(x^*).
\end{equation*}
Define the function $f(x) := \sum_{i = 1}^n -x_i w_i^*$ and choose $\lambda_i^k := 0$, for $i \notin I_g(x^*)$.
Then clearly
$ \nabla f(x^k)_i + w_i^k \to 0,$
and $x^*$ is an AS-stationary point. By assumption, $x^*$ is already S-stationary, which is equivalent to $w^* = -\nabla f(x^*) \in K(x^*)$.
\end{proof}

\noindent 
Having identified AS-regularity as a constraint qualifcation, the
question is how this property is related to other SP-CQs.
Among those given in Definition~\ref{Def:SP-CQs}, the weakest one
is SP-RCPLD. The following result shows that this condition still
implies AS-regularity.

\begin{thm}\label{Thm:SPRCPLD-ASREGULARITY}
Let $x^*$ be feasible for \eqref{SPO} with feasible set $X$
defined by \eqref{Eq:XRep}. Assume 
SP-RCPLD is satisfied at $x^*$. Then AS-regularity holds at $x^*$.
\end{thm}

\begin{proof}
    By construction, SP-RCPLD implies standard RCPLD of for the 
    tightened program (\ref{TNLP}) at $x^*$. It is known from \cite{Andreani2016} that RCPLD yields AKKT-regularity for 
    this program. This, however, is exactly our AS-regularity
    condition.
\end{proof}

\section{An Exact Penalty Algorithm}\label{Sec:Algorithm}

Throughout this section, we assume that all functions $ f, g $, and
$ h $ involved in the sparse optimization problem are at least
continuously differentiable. Then \eqref{SPOcp} represents a smooth 
reformulation of the nonsmooth problem \eqref{SPO}. 
Consequently, the reformulation \eqref{SPOcp} allows the application
of a variety of different methods known from smooth optimization
in order to solve the given sparse problem \eqref{SPO}.

On the other hand, a suitable choice for solving \eqref{SPOcp}
requires some care. For example, taking into account the almost
separable structure of \eqref{SPOcp} in terms of the two variables
$ x $ and $ y $, it is very tempting to apply an alternating
minimization approach to this problem which uses a separate
minimization with respect to the variables $ x $ and $ y $.
This approach has the major advantage that the resulting 
subproblems are (usually) very easy to solve. However, this 
technique then terminates after the first cylce with an
S-stationary point and, afterwards, does not make any further
progress. This is unfortunate since the corresponding objective
function value is typically very poor. In fact, this method
often gets stuck at a local minimum with a relatively large function
value, so that the method terminates with a point that is far
away from being globally optimal.

In view of our experience, and in order to obtain good candidates
for a global minimum, it is advantageous to apply a technique
which might be feasible or approximately feasible with respect
to the standard constraints $ g(x) \leq 0 $ and $ h(x) = 0 $,
but with the complementarity term not approaching zero too fast, 
because this leaves some freedom in reducing the remaining
objective function.

The aim of this section is therefore to present an exact 
penalty approach for the solution of the sparse optimization problem with feasible set $ X $ given by \eqref{Eq:XRep}.
Our exact penalty method is based on the reformulation
\eqref{SPOcp} and penalizes the (difficult) complementarity term
only, whereas the remaining restrictions stay as constraints
in the penalized problem. Using the standard $ \ell_1 $-penalty
function, the penalized objective function then reads
$$ 
   f(x) + p^{\rho}(y) + \alpha \left | x^Ty \right | = f(x) + p^{\rho}(y) + \alpha \norm{x \circ y}_1 .
$$
The $ \ell_1 $-term usually leads to a nonsmoothness of this
penalty approach, which is the major drawback of this technique.
In our particular situation, however, we have sign contraints 
on $x$ and $y$, cf.\ the reformulated problem \eqref{SPOcp}
once again. Hence,  we may remove the absolute value and obtain the
following (smooth!) penalized version of \eqref{SPOref}:
\begin{equation}
    \tag{Pen($\alpha$)} \label{Psq}
        \min_{x,y} \ f(x) + p^{\rho}(y) + \alpha x^T y \quad
        \text{s.t.} \quad g(x) \leq 0, \ h(x) = 0, \ x \geq 0, \ y \ge 0.
\end{equation}
This motivates the following exact penalty-type algorithm.

\begin{algo}~\label{PenaltyAlg} (Exact Penalty Method for Sparse
	Optimization)
	\begin{enumerate}
    \item Choose a non-negative sequence $\varepsilon_k \searrow 0$ and parameters $\alpha_0 > 0$, $\beta > 1$, and $\delta \ge 0$.
    \item For $k = 0, 1, 2, ...$, compute $(x^{k+1},y^{k+1}, \lambda^{k+1}, \mu^{k+1}, \nu^{k+1}_x, \nu^{k+1}_y) \in \RR^n_+ \times \RR^n_+ \times \RR^p_+ \times \RR^m \times \RR^n_+ \times \RR^n_+$ such that
    \begin{align*}
        \lVert \nabla_x L^{SP}(x^{k+1}, \mu^{k+1}, \lambda^{k+1}) - \sum_{i = 1}^n (\nu_x^{k+1})_i e_i + \alpha_k y^{k+1} \rVert \le & \ \varepsilon_k \\
        \lVert \nabla p^\rho(y^{k+1}) - \sum_{i=1}^n (\nu_y^{k+1})_i e_i + \alpha_k x^{k+1}\rVert \le & \ \varepsilon_k\\
        \min |\{-g_i(x^{k+1}), \lambda_i^{k+1}\}| \le & \ \varepsilon_k, \quad i = 1,...,p\\
       \lVert h(x^{k+1}) \rVert \le & \ \varepsilon_k,\\
       \min\{x_i^{k+1},(\nu_x^{k+1})_i\} \le & \ \varepsilon_k, \quad i = 1,...,n \\ \min\{y_i^{k+1}, (\nu_y^{k+1})_i\} \le & \ \varepsilon_k, \quad i = 1,...,n
    \end{align*}
    \item If $$ \varepsilon_k \le \delta \quad \text{and} \quad \sum_{i = 1}^n x_i^{k+1}y_i^{k+1} \le \delta $$ then STOP.  Otherwise set 
    $ \alpha_{k+1} = \alpha_k \cdot \beta$
    and go to 2.    
\end{enumerate}
\end{algo}

\noindent
The main computational burden, of course, is in step 2. Note that
we do not require to solve the penalized subproblems exactly.
The tests within step 2 only check whether we have an approximate 
KKT point of the penalized problem \eqref{Psq}, with 
$ \varepsilon_k $ denoting the measure of inexactness, and with 
$ \lambda^{k+1}, \mu^{k+1}, \nu_x^{k+1} $, and $ \nu_y^{k+1} $ being
the Lagrange multipliers associated to the constraints
$ g(x) \leq 0, h(x) = 0, x \geq 0 $, and $ y \geq 0 $, respectively.
This general framework allows plenty of methods in order to 
deal with the penalized subproblems, which is an important feature
of the overall method since
a suitable choice depends on the particular problem under
consideration. In addition, we note that some
methods might deal with the nonnegativity constraints $ x \geq 0,
y \geq 0 $ explicitly, so that these methods do not generate
corresponding multiplier estimates $ \nu_x^{k+1}, \nu_y^{k+1} $.
In this situation, we can simply delete the two final tests in
step 2, and replace the first two by the related (multiplier-free)
tests
\begin{align*}
        \lVert P_{[0,\infty)}\left(x^{k+1} - (\nabla_x L^{SP}(x^{k+1}, \mu^{k+1}, \lambda^{k+1}) + \alpha_k y^{k+1})\right) - x^{k+1} \rVert \le & \ \varepsilon_k, \\
        \lVert P_{[0, \infty)}\left(y^{k+1} - (\nabla p^\rho(y^{k+1}) + \alpha_k x^{k+1})\right) - y^{k+1} \rVert \le & \ \varepsilon_k.
\end{align*}
The subsequent theory remains true with this kind of test, too.
Our analysis, however, concentrates on the inexactness measure
from step 2.

The final step 3 represents the stopping criterion for the outer
iteration. Based on the termination parameter $ \delta \geq 0 $,
we simply check whether we are (approximately) feasible both 
with respect to the standard constraints and with respect to 
the penalized complementarity constraints. For the (theoretical)
choice $ \delta = 0 $, it follows immediately that we terminate
with a feasible point. The following result shows that, in this
case, we even have an S-stationary point.

\begin{lem}\label{thm:AccPointSstat}
Let $\delta = 0$ and assume that 
Algorithm~\ref{PenaltyAlg} terminates after a finite number of steps in a point $(x^*,y^*).$ Then $(x^*,y^*)$ is a KKT point of $\eqref{SPOref}$, and $x^*$ is S-stationary.
\end{lem}

\begin{proof}
Let $\lambda^*, \mu^*, \nu_x^*, \nu_y^*$ be the associated multipliers 
of the iterate $ (x^*, y^*) $. We first show that
$ x^* $ is S-stationary. Since $ \delta = 0 $ by assumption,
we also have from step 3 that $ \varepsilon_k = 0 $ for the corresponding iteration $ k $. Therefore, step 2 immediately
implies that $ x^* $ is feasible for the sparse optimization
problem \eqref{SPO}, and $ (x^*, \lambda^*) $ satisfies the
complementarity conditions. Now, consider an index $ i $ with
$ x_i^* \neq 0 $. Then steps 2 and 3 yield  $y_i^* = 0$ and $(\nu_x^*)_i = 0$. Using step 2 again, this yields
$ \nabla_{x_i} L^{SP}(x^*, \lambda^*, \mu^*) = 0$. By
definition, this shows that $ x^* $ is an S-stationary point.

Next, consider an index $ i $ with $ x_i^* = 0 $. Then we have
$ \nabla p_i^\rho(y_i^*) - (\nu_y^*)_i = 0 $ from step 2,
which implies $\nabla p_i^\rho(y_i^*) \ge 0$ and hence $y_i^* \ge s_i^\rho > 0$ by convexity of $p_i^\rho$ and uniqueness of $s_i^\rho$. As a result $(\nu_y^*)_i = 0$ and $y_i^* = s_i^\rho$.
Consequently, the vector $ y^* $ satisfies the relation
\eqref{Eq:y-star}. Theorem~\ref{Thm:KKT-Equivalence} therefore
shows that $ (x^*, y^*) $ is a KKT point of \eqref{SPOref}.
\end{proof}

\noindent
We next want to show an exactness result for our penalty
approach. In principle, there are two directions one is interested in,
namely that a stationary point of the penalized problem \eqref{Psq} yields a stationary point of the sparse optimization problem \eqref{SPO}, and vice versa. Exactness results usually concentrate 
on the opposite direction only. In fact, this direction is
usually the easier one and, for our particular setting, contained
in the following result.

\begin{thm}\label{Thm:Stat-1}
Let $(x^*,y^*)$ be a stationary point of \eqref{SPOcp}. Then there exists an $\alpha^* > 0$ such that $(x^*,y^*)$ is a stationary point of \eqref{Psq} for all $\alpha \geq \alpha^*$.
\end{thm}

\begin{proof}
By stationarity of $(x^*,y^*)$ for \eqref{SPOcp}, 
we obtain, after a simple transformation, that
\begin{equation}\label{eqn:SSQstat} 
	0 = \begin{pmatrix} \nabla f(x^*) + h'(x^*)^T \mu^* + g'(x^*)^T \lambda^* + \sum_{I_0(x^*)} \alpha_i^x e_i \\
		\nabla p^\rho(y^*) + \sum_{I_0(y^*)} \alpha_i^y e_i\end{pmatrix}
\end{equation}
holds for some multipliers $\mu^*, \lambda^*, \alpha^x, \alpha^y$ 
with sign constraints only w.r.t.\ to $\lambda^*$, and further 
$y_i^* = s_i^\rho$ if and only if $i \in I_0(x^*)$. On the other
hand, the point $(x^*,y^*)$ is stationary for \eqref{Psq} if there exists multipliers $\mu$, $\lambda \ge 0$, $\nu^x \ge 0$, and $\nu^y \ge 0$ such that
\begin{equation}\label{Eq:Statcons-2}
	0 = \begin{pmatrix} \nabla f(x^*) + h'(x^*)^T \mu + g'(x^*)^T \lambda - \sum_{I_{0}(x^*)} \nu^x_i e_i + \alpha \sum_{I_{0}(x^*)}s_i^\rho e_i \\
		\nabla p^\rho(y^*) - \sum_{I_{0}(y^*)} \nu_i^y e_i + \alpha \sum_{I_{0}(y^*)}x_i^*e_i\end{pmatrix}
\end{equation}
holds. Now, setting $ \mu = \mu^*$, $\lambda = \lambda^*$, 
$\nu_i^x = \alpha s_i^\rho - \alpha_i^x \ (i \in I_0(x^*))$ and $\nu_i^y = \alpha x_i^* - \alpha_i^y \ (i \in I_0(y^*))$
for an arbitrary $ \alpha > 0 $, we see that \eqref{Eq:Statcons-2}
is a direct consequence of \eqref{eqn:SSQstat}. Moreover,
for 
\begin{equation*}
	\alpha \geq \alpha^* := \max \bigg\{ \frac{| \alpha_i^x |}{s_i^{\rho}} \ (i \in I_0(x^*)), \frac{| \alpha_i^y|}{x_i^*} \ (i \in I_0(y^*)) \bigg\},
\end{equation*}
we also have $ \nu^x_i \geq 0 \ (i \in I_0(x^*)) $ and 
$ \nu^y_i \geq 0 \ (i \in I_0(y^*)) $. It follows that 
$ (x^*,y^*) $ is a stationary point of \eqref{Psq} for all
$ \alpha \geq \alpha^* $.
\end{proof}

\noindent
Note that Theorem~\ref{Thm:Stat-1} does not require any
constraint qualification. Typical exactness results of this
kind need an MFCQ-type assumption which, here, is not necessary
for two reasons: First, the (potentially simple) standard
constraints are still in the constraints of the penalized
problem, and second, the penalized (difficult) complementarity
constraint has a very simple structure such that no constraint
qualification is needed to verify the exactness statement 
from Theorem~\ref{Thm:Stat-1}. In fact, this complementarity
constraint alone satisfies automatically any constraint
qualification.

Before presenting an exactness result for the other direction,
we first consider a simple example.

\begin{example}\label{Ex:Exact-Counter}
Consider the one-dimensional sparse optimization problem 
\begin{equation*}
	\min_x \ -x + \| x \|_0 \quad \text{s.t.} \quad x \geq 0
\end{equation*}
(note that we take $ \rho = 1 $ only for the sake of simplicity).
Then $ x^* := 0 $ is both a local minimum and an S-stationary point.
Using the function $ p_i^{\rho} $ from Example~\ref{Ex:p-rho} (b),
the corresponding penalized problem \eqref{Psq} reads
$$ 
    \min_{x,y} \ -x + \frac{1}{2} (y - \sqrt{2})^2 
    + \alpha xy
    \quad \text{s.t.} \quad x \geq 0, \ y \geq 0
$$
with an arbitrary $ \alpha > 0 $. Then, for any $ \alpha \geq 1/\sqrt{2} $, 
$$
   (x,y) := \frac{1}{\alpha}\left(\sqrt{2} - \frac{1}{\alpha}, 1\right)
$$
are stationary points of problem \eqref{Psq}, but
none of these points satisfy the complementarity conditions.
On the other hand, observe that, for $ \alpha \to \infty $, the
corresponding sequence of stationary points converges to
$ (x^*,y^*) = (0,0) $, hence the complementarity 
condition $ x^* y^* = 0 $ is satisfied in the limit. Note, 
however, that this limit point does not satisfy the relation
\eqref{Eq:y-star} which, in particular, guarantees that the 
bi-active set is empty.
\end{example}

\noindent
The following result contains an exactness statement for the other
direction. This result may
be viewed as a generalization of a related theorem given 
in \cite{Ralph2004} in the context of mathematical programs with
equilibrium contraints (MPECs), though our assumptions are weaker.

\begin{thm}\label{Thm:Exact-2}
Let $(x^*,y^*)$ be a stationary point of \eqref{SPOcp}
such that SP-MFCQ holds at $ x^* $. 
Then there exists an $ \alpha^* > 0 $ and a
neighborhood $U$ of $(x^*,y^*)$ such that for all $\alpha \geq \alpha^*$, every stationary point of \eqref{Psq} in $U$ is a stationary point of \eqref{SPOcp}.
\end{thm}

\begin{proof}
Assume, by contradiction, that there is a sequence $ \alpha_k \to \infty$ and a sequence $\{(x^k,y^k)\}$ such that $(x^k,y^k) \to (x^*,y^*)$, where $(x^k,y^k)$ is stationary for \eqref{Psq}
with $ \alpha = \alpha_k $, but
not stationary for \eqref{SPOcp}. Recall that stationarity of $(x^k,y^k)$ for \eqref{Psq} with $ \alpha = \alpha_k $ implies the existence of multipliers $(\mu^k,\lambda^k, \nu_x^k, \nu_y^k)$ 
such that 
\begin{equation}
	\label{eqn:PSQstat}
	0 = \begin{pmatrix}
		\nabla f(x^k) + h'(x^k)^T \mu^k + \sum_{i \in I_g(x^k)} \lambda_i^k \nabla g_i(x^k) + \alpha_k y^k - \sum_{i \in I_{0}(x^k)} (\nu_x^k)_i e_i \\
		\nabla p^\rho(y^k) + \alpha_k x^k - \sum_{i \in I_{0}(y^k)} (\nu_y^k)_ie_i
	\end{pmatrix},
\end{equation}
where $x^k \in X$ and $y^k \ge 0$. We divide the proof into
several steps. \medskip 

\noindent 
(a) Since $ (x^*, y^*) $ is a stationary point of \eqref{SPOcp},
we recall from Remark~\ref{Rem:Biactive-empty} that 
the biactive set $ \mathcal{B} (x^*,y^*) = \{ i \, | \ (x_i^*,y_i^*)
= (0,0) \} $ is empty. We stress that this observation
plays a central role for the subsequent proof. \medskip

\noindent 
(b) Since $ x^k \to x^* $, the continuity of $ g $ implies that 
$ I_g(x^k) \subseteq I_g(x^*) $ for all $ k $ sufficiently large.
Because $ \lambda_i^k = 0 $ for all $ i \not\in I_g(x^k) $, we 
may therefore replace the index set $ I_g(x^k) $ in \eqref{eqn:PSQstat} 
with the constant set $ I_g(x^*) $ for all $ k $ large enough.
Hence, we have both 
\begin{equation}\label{eqn:PSQstat-1}
	\nabla f(x^k) + h'(x^k)^T \mu^k + \sum_{i \in I_g(x^*)} \lambda_i^k \nabla g_i(x^k) + \alpha_k y^k - 
	\sum_{i \in I_{0}(x^k)} (\nu^k_x)_i e_i= 0
\end{equation}
and 
\begin{equation}\label{eqn:PSQstat-2}
	\nabla p^\rho(y^k) + \alpha_k x^k - 
	\sum_{i \in I_{0}(y^k)} (\nu_y^k)_i e_i = 0
\end{equation}
for all $ k $ sufficiently large. \medskip 

\noindent 
(c) We claim that $ y_i^k = 0 $ holds for all $ i \in I_0(y^*) $
and all $ k $ sufficiently large. To this end, assume there is
an index $ i \in I_0(y^*) $ and a subsequence $ \{ y_i^k \}_K $
such that $ 0 < y_i^k \to_K y_i^* = 0 $. It then follows that
$ (\nu_y^k)_i = 0 $ for all $ k \in K $. Hence \eqref{eqn:PSQstat-2}
implies 
\begin{equation*}
	0 = \nabla p^{\rho} (y_i^k) + \alpha_k x_i^k \quad \forall k \in K.
\end{equation*}
On the other hand, the right-hand side tends to $ + \infty $ for 
$ k \to_K \infty $ since $ y_i^k \to_K y_i^* $ and the continuity
of $ \nabla p^{\rho} $ implies the convergence of the first term,
whereas the second term is unbounded since $ \alpha_k \to_K \infty $
and $ x_i^k \to x_i^* $ with $ x_i^* > 0 $ due to (a). \medskip 

\noindent 
(d) We claim that
\begin{equation}\label{Eq:Proof-d}
	\bigg[ \nabla f(x^k) + h'(x^k)^T \mu^k + \sum_{i \in I_g(x^*)}
	\lambda_i^k \nabla g_i(x^k) \bigg]_i = 0 \quad 
    \forall i \in I_0(y^*)
\end{equation}
for all $ k $ sufficiently large. This follows from \eqref{eqn:PSQstat-1} together with the fact that $ y_i^k = 0 $
for all $ i \in I_0(y^*) $ and $ k $ sufficiently large by part (c),
and since $ I_0(x^k) \subseteq I_0 (x^*) $ with $ I_0(x^*) \cap
I_0(y^*) = \emptyset $ by part (a). \medskip

\noindent 
(e) In view of (a) and \eqref{Eq:Proof-d}, we can find scalars 
$ \gamma_i^k \in \mathbb{R} $ for $ i \in I_0(x^*) $ such that 
\begin{equation*}
   \nabla f(x^k) + h'(x^k)^T \mu^k + \sum_{i \in I_g(x^*)}
   \lambda_i^k \nabla g_i(x^k) + \sum_{i \in I_0(x^*)} 
   \gamma_i^k e_i = 0
\end{equation*}
holds for all $ k $ large enough. Due to the assumed SP-MFCQ
condition, a standard argument then shows that the sequence 
of multipliers
$ \big\{ \big( \lambda_i^k ( i \in I_g(x^*)), \mu^k,
\gamma_i^k (i \in I_0(x^*)) \big) \big\} $ remains bounded. \medskip

\noindent 
(f) We also claim that $ x_i^k = 0 $ for all $ i \in I_0(x^*) $ 
and all $ k $ sufficiently large. Assume, by contradiction, that
there is a subsequence $ \{ x_i^k \}_K $ with $ 0 < x_i^k \to_K 
x_i^* $. Then $ (\nu_x^k)_i = 0 $ holds for all $ k \in K $.
Consequently, we obtain from \eqref{eqn:PSQstat-1} that 
\begin{equation*}
	0 = \nabla_{x_i} L^{SP} (x^k,\lambda^k, \mu^k) + \alpha_k y_i^k.
\end{equation*}
Now, the first term on the right-hand side remains bounded by
continuity of $ \nabla_{x_i} L^{SP} $ as well as the fact
that $ x^k \to x^* $ and the boundedness of the multiplier
sequences $ \{ \lambda^k \} $ and $ \{ \mu^k \} $, cf.\ part (e).
On the other hand, the second term converges to $ + \infty $ since
$ \alpha_k \to \infty $ and $ y_i^k \to_K y_i^* > 0 $ for $ i 
\in I_0(x^*) $, see part (a). \medskip

\noindent 
(g) In view of parts (c) and (f), we, in particular, have 
$ x_i^k y_i^k = 0 $ for all $ i = 1, \ldots, n $ and all $ k $
sufficiently large. This shows that $ (x^k,y^k) $ is, at least,
feasible for \eqref{SPOcp}. Furthermore, since $ (x^k,y^k) \to 
(x^*,y^*) $ and $ \mathcal{B} (x^*,y^*) = \emptyset $ by part (a),
we also have $ \mathcal{B} (x^k,y^k) = \emptyset $ for all 
$ k $ large enough. We can therefore define the multipliers
\begin{equation*}
	(\alpha_i^x)^k := \alpha_k y_i^k - (\nu_x^k)_i \quad 
	(i \in I_0(x^k)) \quad \text{and} \quad 
	(\alpha_i^y)^k := \alpha_k x_i^k - (\nu_y^k)_i \quad 
	(i \in I_0(y^k)),
\end{equation*}
so that the stationary conditions of the penalized problem yield
\begin{equation*}
	0 = \begin{pmatrix}
		\nabla f(x^k) + h'(x^k)^T \mu^k + \sum_{i \in I_g(x^k)} \lambda_i^k \nabla g_i(x^k) + \sum_{I_0(x^k)} (\alpha^x_i)^k e_i \\
		\nabla p^\rho(y^k) + \sum_{I_0(y^k)} (\alpha^y_i)^k e_i
	\end{pmatrix}
\end{equation*}
for all $ k $ sufficiently large, cf.\ \eqref{eqn:PSQstat}.
Altogether, this implies that $ (x^k, y^k) $ is a stationary
point of \eqref{SPOref} for all $ k $ sufficiently large, and this
contradiction completes the proof
\end{proof}

Observe that part (d) of the previous proof already shows
that $ x^k $, for all $ k $ sufficiently large, is an
S-stationary point of \eqref{SPO}, and that this holds
without any constraint qualification. The SP-MFCQ assumption
is mainly used to show that the pair $ (x^k,y^k) $ is
eventually feasible for \eqref{SPOcp}, i.e., satisfies the
complementarity condition $ x \circ y = 0 $.

Note further that Example~\ref{Ex:Exact-Counter} does not contradict
the statement of Theorem~\ref{Thm:Exact-2}. Though SP-MFCQ
holds for this example in $ x^* = 0 $, the sequence of stationary
points of the corresponding penalized problems converges to 
$ (0,0) $, which is not a stationary point \eqref{SPOcp}, 
as assumed in Theorem~\ref{Thm:Exact-2}.

The following result provides a relation between the second-order
condition of the penalized problem \eqref{Psq} and SP-SOSC 
for the sparse optimization problem.

\begin{thm}
Let $(x^*,y^*)$ be stationary for \eqref{Psq} and feasible for \eqref{SPOcp}. Assume that $p^\rho$ is twice continuously differentiable and that standard SOSC holds at \eqref{Psq}.
Then SP-SOSC holds at $x^*$.
\end{thm}

\begin{proof}
An elementary calculation shows that standard SOSC for \eqref{Psq} is given by
\begin{align}
    \begin{pmatrix}
        dx \\ dy
    \end{pmatrix}^T 
    \begin{pmatrix}
        \nabla_{xx}^2 L^{SP}(x^*,\mu^*,\lambda^*) & \alpha I \\ \alpha I & \nabla_{yy}^2 p^\rho(y)
    \end{pmatrix}
    \begin{pmatrix}
        dx \\ dy
    \end{pmatrix}  > 0 \notag\\
    \Longleftrightarrow \quad (dx)^T \nabla_{xx}^2 L^{SP}(x^*,\mu^*,\lambda^*) dx + 2 \alpha (dx)^T dy + (dy)^T \nabla_{yy}^2 p^\rho(y) dy > 0, \label{eqn:PSQSOSC2}
\end{align}
for all $(dx, dy) \neq (0, 0)$ such that
\begin{align}
    dx_i &= 0, \ (i \in I_0(x^*), \ \nu_i^x > 0), \nonumber \\
    dx_i &\ge 0, \ (i \in I_0(x^*), \ \nu_i^x = 0), \nonumber \\
    \nabla h_i(x^*)^T dx &= 0, \ (i = 1,...,m) \label{cnd3}, \\
    \nabla g_i(x^*)^T dx &= 0, \ (i \in I_g(x^*), \ \lambda^*_i > 0) \label{cnd4},\\
    \nabla g_i(x^*)^T dx &\le 0, \ (i \in I_g(x^*), \ \lambda^*_i = 0) \label{cnd5},\\
    dy_i &= 0, \ (i \in I_0(y^*), \nu_i^y > 0), \nonumber \\
    dy_i &\ge 0, \ (i \in I_0(y^*), \nu_i^y = 0), \nonumber
\end{align}
where $\nu_i^x, \nu_i^y$ and $\lambda^*_i$ are the Lagrangian multipliers associated to the sign constraints on $x$ and $y$ and to the inequality constraints governed by $g$, respectively. Now, choose $dy = 0$ and $dx$ such that $dx_i = 0$ for all $i \in I_0(x^*)$ and conditions (\ref{cnd3}), (\ref{cnd4}) and (\ref{cnd5}) hold. The claim then follows directly from (\ref{eqn:PSQSOSC2}).
\end{proof}
    
A difficulty with (exact) penalty approaches for general optimization
problems is that accumulation points are not guaranteed to be feasible.
In our setting, this feasibility issue arises for the complementarity
constraints only, and it turns out that, due to the particular
structure of our reformulated problem, these complementarity
conditions are satisfied even if Algorithm~\ref{PenaltyAlg}
does not terminate after finitely many iterations.

\begin{thm}\label{Thm:Feasibility}
Let $\delta = 0$, and let $ \{ (x^{k}, y^{k}) \} $ be an infinite
sequence generated by Algorithm~\ref{PenaltyAlg} such that 
$x^{k+1} \to_K x^*$ on some subsequence $K$. Then there is a subset $K'\subseteq K$ such that $y^{k+1} \to_{K'} y^*$ and 
$ x_i^* y_i^* = 0$ holds for all $i = 1,...,n$.
\end{thm}

\begin{proof}
We first show that the corresponding subsequence $ \{ y^{k+1} \}_K $
remains bounded. By contradiction, assume that there is an index
$ i $ such that $ \{ y_i^{k+1} \} $ is unbounded. Due to the 
nonnegativity constraint, we may therefore assume, without loss
of generality, that $ y_i^{k+1} \to_K \infty $. In particular,
we then have $y_i^{k+1} \geq 2 \cdot s_i^\rho$ for infinitely many $k
\in K $. The convexity of $ p_i^{\rho} $ then implies
$$ 
   \nabla p_i^\rho(y_i^{k+1}) \geq \nabla p_i^\rho(2s_i^\rho)=:c > \nabla p_i^\rho(s_i^\rho) = 0.
$$ 
In particular, we then have
$$ 
   \nabla p_i^\rho(y_i^{k+1}) + \alpha_k x_i^{k+1} \ge c
$$
and therefore $(\nu_y^{k+1})_i \geq c/2$ for infinitely many $k \in K$
due to the second termination check in step 2. On the other hand, by the final condition in step 2, we have
$$
  \min\{y_i^{k+1},(\nu_y^{k+1})_i\} \to 0,
$$
and this contradiction shows that $ \{ y^{k+1} \}_K $ is indeed
a bounded sequence.

Consequently, there is a subset $K'\subseteq K$ such that 
$ \{ y^{k+1} \}_{K'} $ converges to some point $ y^* $. We
claim that $ x_i^* y_i^* = 0 $ holds for this limit for all
$ i = 1, \ldots, n $. For $ x_i^* = 0 $, there is nothing to
prove. Hence consider an index $ i $ with $ x_i^* > 0 $.
Then clearly $\alpha_k x_i^{k+1} \to_K \infty$. 
Since $ p_i^{\rho} $ is convex by assumption, its derivative 
$\nabla p_i^\rho$ is monotone. Taking into account the
sign restriction $y_i^{k+1}\ge 0$, we therefore obtain
$$
   \nabla p_i^\rho(y_i^{k+1}) \ge \nabla p_i^\rho(0).
$$
This implies
$$ 
   \nabla p_i^\rho(y_i^{k+1}) + \alpha_k x_i^{k+1} \to_{K'} \infty,
$$
and the second termination check in step 2 of Algorithm~\ref{PenaltyAlg}
therefore yields
$$ 
   (\nu_y^{k+1})_i \to_{K'} \infty.
$$
Hence, the sixth condition in step 2 immediately gives $y_i^{k+1} \to_{K'} 0 = y_i^*$, and this completes the proof.
\end{proof}

The above theorem shows that every accumulation point of 
Algorithm~\ref{PenaltyAlg} is indeed (approximately) feasible for (\ref{SPOref}). Note that this is not as surprising as it seems
in the beginning. In fact, when looking at the original problem \ref{SPO}, the feasible set is given by $ X $ and depends
on the variables $ x $ alone. Moving to an auxiliary variable $y$ should not increase the difficulty to find feasible points for the reformulation. 

In general, we cannot guarantee to obtain approximate stationary points if the algorithm does not terminate after a finite number of iterations. We may, however, choose $\alpha_k$ and $\varepsilon_k$ in dependence to recover such a result.

\begin{thm}\label{thm:AS-stat-conv}
Let $\delta = 0 $, and let $ \{ (x^{k}, y^{k}) \} $ be an infinite
sequence generated by Algorithm~\ref{PenaltyAlg} such that $x^{k+1} \to_K x^*$ on a subsequence $K$. Then the following statements hold:
\begin{enumerate}
 \item[(a)] If $y_i^{k+1} \alpha_k \to_{K} 0$ for all $i \notin I_0(x^*)$, then $x^*$ is an AS-stationary point.
 \item[(b)] If $\varepsilon_k \alpha_k \to 0$, then $y^{k+1}_i \alpha_k \to 0$ for all $i \notin I_0(x^*)$.
\end{enumerate}
\end{thm}

\begin{proof}
(a) First recall that $ \{ x^{k+1} \}_K \to x^* $
by assumption, and that $ \min \big\{ - g_i (x^{k+1}), \lambda_i^{k+1})\} \to_K 0 $ follows from the third test
in Algorithm~\ref{PenaltyAlg}. Hence, it remains to show that 
\begin{equation}\label{Eq:AS-Stat-holds}
	\nabla_{x_i} L^{SP} (x^{k+1}, \lambda^{k+1}, \mu^{k+1}) \to_K 0
	\quad (i \notin I_0 (x^*))
\end{equation}
holds. Therefore, consider an arbitrary index $ i \notin I_0(x^*) $,
so that $ x_i^* > 0 $. Then it follows from the fifth
test in step 2 that $ \{ (\nu_x^{k+1})_i \} \to_K 0 $.
Together with the assumption $ y_i^{k+1} \alpha_k \to_{K} 0$, we
see that \eqref{Eq:AS-Stat-holds} follows from the first test
in step 2. \medskip

\noindent
(b) Consider an index $ i \notin I_0(x^*) $, so that $ x_i^* > 0 $.
We then have $ \alpha_k x_i^{k+1} \to_K \infty $. Then the second
condition in step 2 of Algorithm~\ref{PenaltyAlg} implies $(\nu_y^{k+1})_i \to_K \infty$ since $\nabla p^\rho_i$ takes its
(finite) minimum at $0$. Multiplying condition 6 by $\alpha_k$ yields
$$ 
   \min \{ \alpha_k y_i^{k+1}, \underbrace{\alpha_k (\nu_y^{k+1})_i}_{\to \infty} \} \leq \alpha_k \varepsilon_k \to 0
$$
and hence $\alpha_k y_i^{k+1} \to 0$. This completes the proof.
\end{proof}

Statement (a) of Theorem~\ref{thm:AS-stat-conv} provides a condition
under which an arbitrary limit point of the sequence $ \{ x^k \} $
is an AS-stationary point and, hence, an S-stationary point under
any of the SP-type constraint qualifications discussed in 
Section~\ref{Sec:AS-Stationarity}. Statement (b) then gives a
sufficient condition under which the assumption from part (a)
holds. Note that this sufficient condition can be realized. 
In fact, in iteration $ k $, we have a penalty parameter
$ \alpha_k $, and then choose a termination parameter $ \varepsilon_k $
such that $ \varepsilon_k = o (1 / \alpha_k) $ holds. From 
a practical point of view, however, this means that we might have to 
choose $ \varepsilon_k $ small, possibly even at a relatively early
stage of the entire method, hence it is not clear whether such
a choice is always desirable.

\section{Numerical Experiments}\label{Sec:Numerics}

The aim of this section is to present a variety of
applications where our exact penalty approach can be applied to.
We recall that all these applications are extremely difficult
due to the $ \ell_0 $-term (in the original formulation of
the sparse optimization problem), nevertheless, the numerical
results indicate that we find very good candidates for a solution
of the underlying problem, quite often even the global minimum.

\subsection{Sparse Portfolio Optimization}

\subsubsection{Preliminaries}

The sparse portfolio optimization can be stated in the form
\begin{equation} 
	\min_x\ x^TQx + \rho \norm{x}_0 \quad \text{s.t.} \quad e^Tx = 1, \ \mu^Tx \ge s, \ 0 \leq x \leq u \label{SPPF},
\end{equation}
with a positive (semi-) definite covariance matrix $Q$,
$ \mu \in \mathbb{R}^n $ the mean of $ n $ possible assets,
$ s > 0 $ the minimum amount of (expected) return, $e=(1,1,...,1)^T$,
and $ u_i $ an upper bound for the variable $ x_i $ which 
represents the percentage of our total investment into asset $ i $.
Hence, the economic interpretation of the portfolio model
\eqref{SPPF} is, basically, as follows: The customer is willing to spend a certain amount of money in a few (due to the 
$ \ell_0 $-term) possible assets in such a way that he minimizes the risk (represented by the objective function) and has at least 
a minimum return. Note that adding an $ \ell_1 $-term instead of
an $ \ell_0 $-term does not yield any sparsity due to the constraints
of this problem.

Note that \eqref{SPPF} can also be written as a mixed integer quadratic program. In fact, using an auxiliary variable $z$, 
problem \eqref{SPPF} is equivalent (in terms of global solutions) to
\begin{equation}\label{MIPSPPF} 
	\min_{(x,z)} \ x^T Q x + \rho e^Tz, \quad \text{s.t.} \quad e^Tx = 1, \ \mu^Tx \ge s, \ u \circ z \ge x \ge 0, \ z \in \{0,1\}^n,
\end{equation}
cf.\ \cite{Bienstock1996}. Since there exists commercially available software to tackle these types of problems like CPLEX or Gurobi, it is,
in principle, possible to find the global minimum. This, in turn,
allows to compare the quality of solutions obtained by our
exact penalty technique.

It is useful to point out that the constraints in (\ref{SPPF}) are all polyhedral. This directly implies for the SP-RCPLD and therefore AS-regularity to hold. 

Application of our technique yields the 
corresponding penalized problem \eqref{Psq}
\begin{equation}\label{PsqSPPF}
        \min_{x,y} \ x^TQx + p^\rho(y) + \alpha x^Ty \quad  \text{s.t.}
        \quad e^Tx = 1, \ \mu^Tx \ge s, \  0 \le x \le u, \ 0 \le y.
\end{equation}
If we follow Example~\ref{Ex:p-rho} (b) and 
choose $p^\rho (y):= 1/2 \norm{y - \sqrt{2\rho}e}_2^2$, 
we may rewrite \eqref{PsqSPPF} in the form
\begin{align}\label{PortfolioQuadProg}
        \min_{x,y} \ \begin{pmatrix}x \\ y\end{pmatrix}^T\hat{Q}\begin{pmatrix}x \\ y\end{pmatrix} - \sqrt{2\rho} e^Ty \quad \text{s.t.} \quad e^Tx = 1, \ \mu^T x \ge s, \ 0\le x \le u, \ 0 \le y,
\end{align}
where
$$ 
    \hat{Q} = \frac{1}{2} \begin{pmatrix}
    2Q & \alpha I \\ \alpha I & I
    \end{pmatrix}.
$$
Note that this is a quadratic program with a positive definite Matrix $\hat{Q}$ as long as $\alpha <\sqrt{2 \lambda_{\min}}$, where 
$ \lambda_{\min} := \min \{ \lambda \in \sigma(Q)\} $ denotes the
minimum eigenvalue of the matrix $ Q $. Hence, for this choice
of $\alpha$, solutions of \eqref{PsqSPPF} are unique and easy
to compute, but not necessarily feasible for \eqref{SPOref}. 
However, this motivates an initial choice $\alpha_0 = \sqrt{2 \lambda_{\min}} \cdot c$, $c \in (0,1)$ for the penalty parameter. 
In our set of test instances, all $\lambda_{\min}$ happened to be strictly positive. We further stress that, in case of a positive definite matrix $Q$, it is easy to see from equation \eqref{eqn:PSQSOSC2} that, at a feasible point $(x^*,y^*)$, SOSC is satisfied for \eqref{SPOref}.

\subsubsection{Numerical Test}

For our numerical tests, we have chosen an instance of 
problems provided by Frangioni and Gentile\footnote{\url{http://groups.di.unipi.it/optimize/Data/MV.html}},
with the constraints
$
   (1 - y_i) l_i \le x_i \le (1-y_i) u_i,
$
corresponding to $x_i = 0$ or $x_i \in [l_i, u_i]$, relaxed to
$$ 
   0 \le x_i \le (1-y_i) u_i.
$$
We first applied the branch-and-bound type algorithm by Gurobi\footnote{\url{https://www.gurobi.com/solutions/gurobi-optimizer/}} to the mixed-integer reformulation \eqref{MIPSPPF} to approximate 
a global minimum. In accordance to the previous subsection, we then
took $p^\rho(y) = 1/2 \norm{y - \sqrt{2 p}e}_2^2$, $\rho = 1$, and computed via Python $\lambda_{\min}$ and $x^0$ as solution to the quadratic programm \eqref{PortfolioQuadProg} for an $\alpha$ just below $\sqrt{2\lambda_{\min}}$ with a call to the 
corresponding Gurobi-module. 
This process took up 21 seconds of CPU-time for all of the 90 test instances.
\begin{figure}[H]
    \centering
    \begin{subfigure}{\linewidth}
    \centering
    \includegraphics[width =\linewidth]{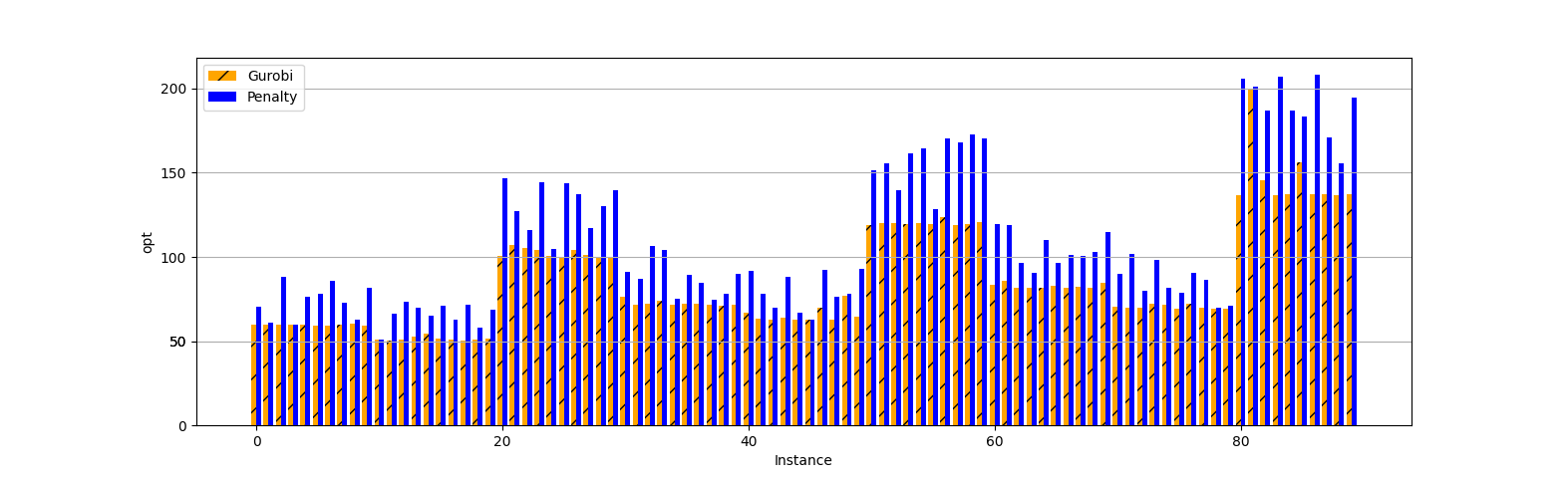}
    \subcaption{Portfolio results by Gurobi and the penalty approach with a Huber-type $p^\rho$ and $\beta=1.1$}
    \label{fig:GurobiVsPenaltyHuber}
    \end{subfigure}
    \begin{subfigure}{\linewidth}
    \centering
    \includegraphics[width = \linewidth]{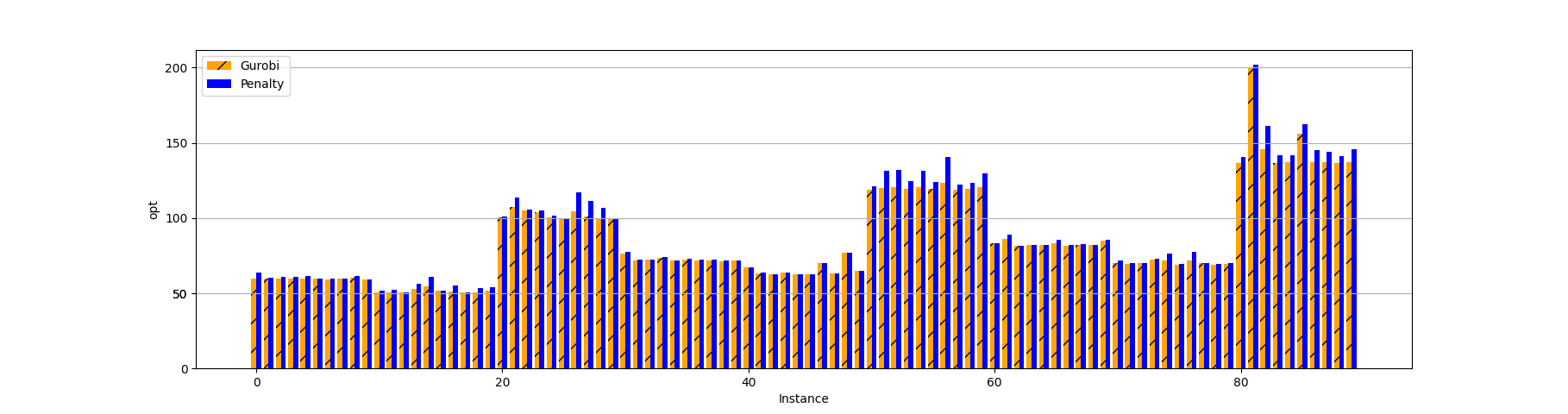}
    \subcaption{Portfolio results by Gurobi and the penalty approach with a smooth $p^\rho$ and $\beta = 5$.}
    \label{fig:GurobiVsPenalty(b)}
    \end{subfigure}
        \begin{subfigure}{\linewidth}
    \centering
    \includegraphics[width = \linewidth]{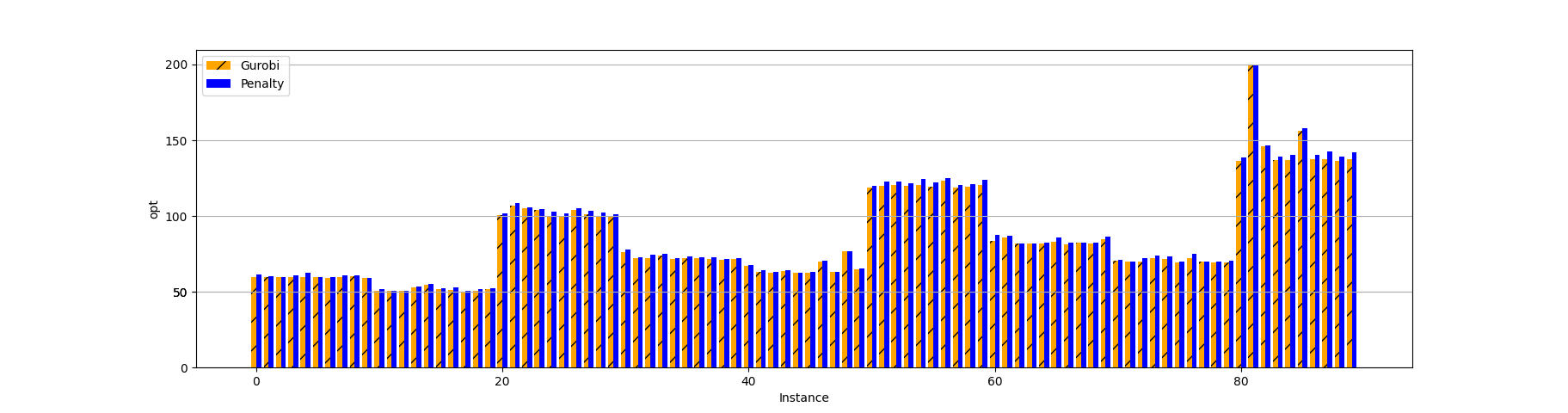}
    \subcaption{Portfolio results by Gurobi and the penalty approach with a smooth $p^\rho$ and $\beta = 2$.}
    \label{fig:GurobiVsPenalty(c)}
    \end{subfigure}
        \begin{subfigure}{\linewidth}
    \centering
    \includegraphics[width = \linewidth]{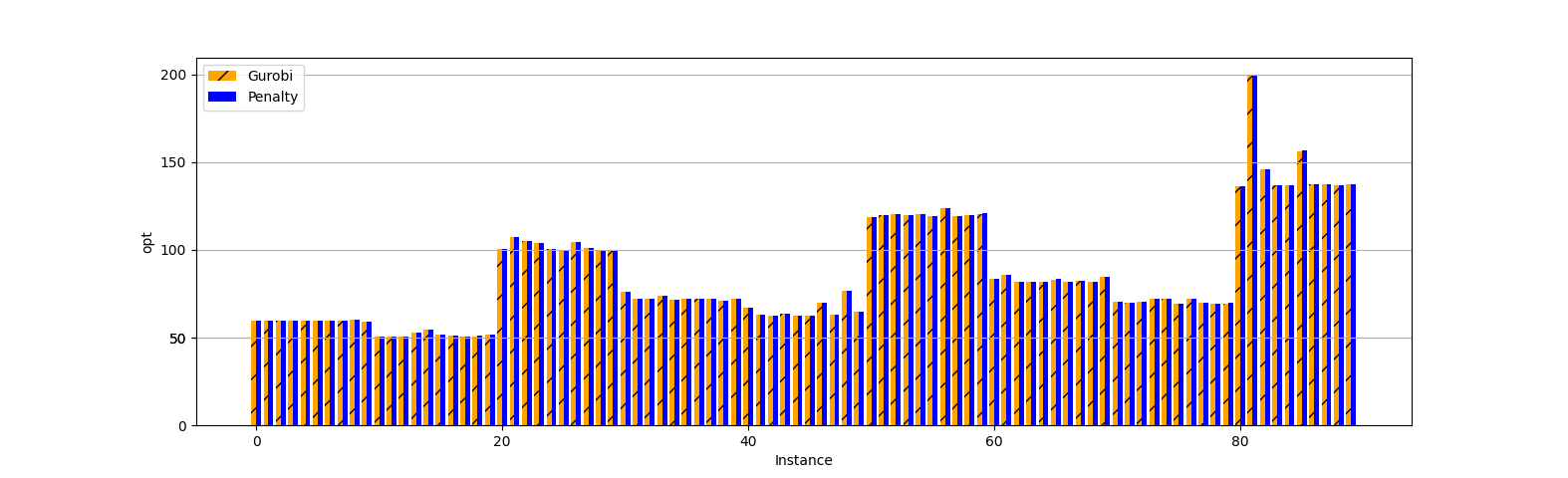}
    \subcaption{Portfolio results by Gurobi and the penalty approach with a smooth $p^\rho$ and $\beta = 1.1$.}
    \label{fig:GurobiVsPenalty(d)}
    \end{subfigure}
    \caption{Overview of the sparse portfolio tests.}
\end{figure}

Afterwards we used our exact penalty approach by computing a solution of \eqref{PsqSPPF} via Algencan under Fortran, with the initial choice  of $x^0$ as described above, and 
$$
   y^0 := \sqrt{2\lambda_{\min}},\quad  
   \alpha_0 := \sqrt{2\lambda_{min}} / 0.95, \quad \beta \in \{1.1, 2, 5\}.
$$
Under Algencan, the subproblems in step two of 
Algorithm~\ref{PenaltyAlg} were solved to an accuracy of $10^{-6}$ with specifically a tolerance of $10^{-8}$ in  feasibility.  The execution was halted once the solution $(x,y)$ provided by Algencan would fulfill the condition $$ x^Ty \le 10^{-6}.$$

With decreasing choices of $\beta$ we are able to recover the global solution given by Gurobi (compare plots in Figures~\ref{fig:GurobiVsPenalty(b)} - \ref{fig:GurobiVsPenalty(d)}). In fact, the results are already
very promising for the penalty updating factors $ \beta = 10 $ and,
especially,
$ \beta = 5 $, and for $ \beta = 1.1 $, the optimal function 
values computed by our exact penalty scheme coincides with the 
optimal function values provided by Gurobi for \emph{all} 90
instances. Regarding the CPU-time for $\beta = 1.1$,  the total time
for the computation of the 30 instances with dimension 200 was around 4.1 seconds, for dimension 300 at 17.5 seconds, and for 
dimension 400 at 25 seconds.

In a second run, we replaced $p^\rho(y)$ via a piecewise Huber-type function in the sense that
$$ p^\rho_i(y) = \xi \cdot \begin{cases}
    \varepsilon(y - \sqrt{2\rho} - \varepsilon) + \frac{1}{2}\varepsilon^2, \quad &y > \rho + \varepsilon,\\
    \frac{1}{2} (y_i - \sqrt{2\rho})^2, \quad &y \in [\rho - \varepsilon, \rho + \varepsilon], \\
    -\varepsilon(y - \sqrt{2\rho} + \varepsilon) + \frac{1}{2}\varepsilon^2, \quad &y < \rho -
    \varepsilon,
\end{cases} \qquad \xi = \frac{\rho}{\varepsilon\sqrt{2\rho} - \frac{1}{2}\varepsilon^2}$$
and set $\rho = 1$, $\varepsilon = 0.1$. 
The call to Algencan yields the results in 
Figure~\ref{fig:GurobiVsPenaltyHuber}.

As we can see, we did overall not recover the global solution in this case, however the computation was in fact shortened to 2.76 seconds for dimension 200, 7.95 seconds for dimension 300 and 12.73 seconds for dimension 400.

\subsection{Sign-constrained Basis Pursuit}

\subsubsection{Preliminaries}

In general, the aim is to find a sparse vector $x$ that approximately satisfies
$$ Ax \approx b.$$
In problem settings as, for instance, mass spectromety as described in \cite{vandenBerg2011}, it is also necessary to introduce the sign constraints $x \ge 0$. We therefore arrive at the formulation
\begin{equation}\label{SCBP}
    \min_x \ \norm{x}_0 \quad \text{s.t.} \quad \norm{Ax - b}_2^2 \le \varepsilon, \ x \ge 0,
\end{equation}
to which the penalty formulation with the choice of $p^\rho(y) = 1/2 \norm{y - \sqrt{2\rho}e}_2^2$ is given by
\begin{equation}\label{PsqSCBP}
    \min_x \ p^{\rho}(y) + \alpha x^Ty \quad \text{s.t.} \quad \norm{Ax - b}_2^2 \le \varepsilon, \ x \ge 0, \ y \ge 0.
\end{equation}
Observe that \eqref{SCBP} is a sparse optimization problem of the form $f(x) + \norm{x}_0$ with $f \equiv 0$. These kind of problems
are particularly challenging since a simple inspection shows that every
feasible point is already a local minimum, which, of course, is also reflected by our stationarity conditions. Nevertheless, we hope that an accumulation point $(x^*,y^*)$ of a sequence $(x^k,y^k)$ produced by the exact penalty approach is feasible and has a convincing sparsity pattern. 

We will construct these problems by choosing a matrix $A$ and a suitable sparse vector $x^0 \ge 0$ such that
$$ 
   b = Ax^0 + r, \quad \varepsilon = \norm{r}_2^2 \cdot (1 + \delta)
$$
with a random vector $r$ and a small $\delta > 0$. By continuity, the interior of the set given by the constraints in \eqref{PsqSCBP} is nonempty. Hence, the Slater-CQ is fulfilled and the penalized formulation always admits Lagrangian multipliers. Furthermore, let $x^*$ be given such that $I_0(x^0) = I_0(x^*)$. Consider then the TNLP for problem \eqref{SCBP} around a feasible point $x^*$:
\begin{align*}
\min_x \ 0 \quad
\text{s.t.} \quad \norm{Ax - b}_2^2 \le \varepsilon, \
x_i = 0, \ i\in I_0(x^*).
\end{align*}
Clearly, the point $x^0$ is also feasible for the above problem and, furthermore, strictly satisfies the inequality constraints. Hence, the Slater-CQ is also satisfied for the TNLP and as an inference $x^*$ is an S-stationary point.

\subsubsection{Numerical Tests}

We tested 200 instances in which we initialized $A$ as a random $\{0,...,99\}^{128 \times 512}$ matrix and chose an original signal $\overline{x}$, with $\overline{x}_i$ identically distributed on the interval $[0,1]$. Afterwards, a support of size $16$ was taken at random, so that
$$ 
   \norm{\overline{x}}_0 = 16.
$$
The vector $b:= A \overline{x}$ was distorted by a Gaussian noise $r \in \RR^{128}$ of mean $0$ and variance $0.5$. We chose $\varepsilon$ such that
$$ 
   \varepsilon > 1.1 \cdot \max_{i = 1,...,200}(0.5 \cdot \norm{r_i}_2)^2,
$$
where $r_i$ denotes the error vector for instance $i$. Furthermore, $x^0$ was initialized as the zero vector, $\rho = 1$, $y^0 = e$, $\alpha^0 = 1.0$ and $\beta = 1.1$. Stopping parameters for Algencan where chosen as before, with the tolerance for approximate stationarity of $10^{-6}$ and for feasibility of $10^{-8}.$
The corresponding numerical results are presented in 
Figure~\ref{fig:CSResult}. Note, in particular, that the
sparsity level generated by our method is, for all instances,
at least as good as the initial guess $ \bar{x} $, and even better
for a number of test problems.

\begin{figure}[ht]
    \centering
    \includegraphics[width = .5\textwidth, height = 6cm]{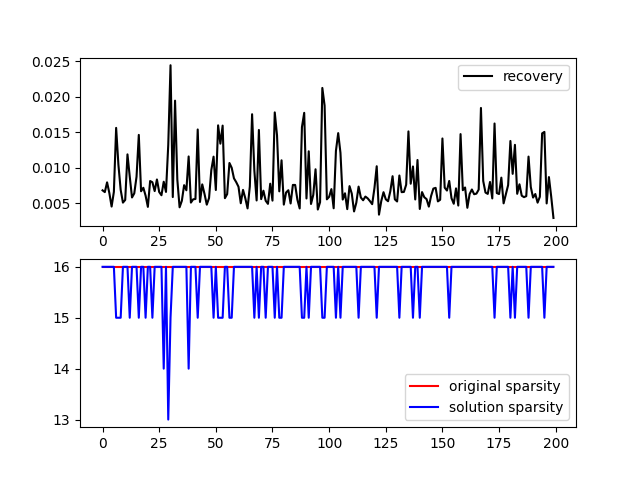}
    \caption{Measurement of the recovery $\norm{\overline{x} - x^s}$ and comparison of the sparsity between the original vector $\overline{x}$ and sparsity of the solution $x^s$.}
    \label{fig:CSResult}
\end{figure}

\subsection{Logistic Regression}

\subsubsection{Problem Definition}

Assume one is interested to train a decision-making algorithm based on probabilities, where we have $m$ data points $(z_i,t_i)$, with $z_i \in \RR^n, \ t_i \in \{1,-1\}^n$, and are looking for a model $p$ with parameters $a = (a_1,...,a_n)$ such that
$$ 
   p(a; z_i) \approx t_i \quad \forall i = 1, \ldots, m.
$$
As a common approach, one chooses the sigmoid function
$$ 
   p(a; z_i) = \frac{1}{1 + \exp(-a^T z_i)}
$$
in order to find a suitable and sparse parameter vector $a$ by
solving the optimization problem
$$ 
   \min_a \frac{1}{m} \sum_{i = 1}^m \log(1 + \exp(-y_iz_i^Ta)) + \rho \norm{a}_0 \st -r \le a \le r
$$
for some large $r$ to guarantee the solvability of the problem.
Note that, here we are lacking the sign constraints on $a$. 
To deal with this problem, we separate $a$ into its positive and negative parts in the sense that
$$ 
   a = a^+ - a^-, \quad a^+ \ge 0, \ a^- \ge 0,
$$
so that our newly found optimization problem is of the form
$$ 
   \min_{a^+, \ a^-} \ f(a^+ - a^-) + \rho \norm{(a^+,a^-)}_0, 
   \quad r\ge a^+ \ge 0, \ r\ge a^- \ge 0.
$$
In general, this approach comes with some drawbacks as the split is not unique and increases the number of local minima. Consider for instance the problem
$$ 
   \min_x \ x^2 + \norm{x}_0, \quad x \in \RR.
$$
Then, clearly, $ x^* = 0 $ is the only local and global minimum.
If, however, we introduce the split, we obtain the formulation
$$ 
   \min_{x^+,\ x^-} \ (x^+ - x^-)^2 + \norm{(x^+,x^-)}_0, \quad x^+ \ge 0, \ x^- \ge 0,
$$
with $(x^+,x^-) = (0,0)$ still being the unique global minimum, but
with each $(x^+,x^-) = \lambda \cdot (1,1)$ being a local minimum for each $\lambda \in \RR$.

Unfortunately, this is not the only problem as we naturally
increase the number of variables and, thus, also the required computational power. Nevertheless, the following subsection shows that this approach still works quite well in practice.

\subsubsection{Numerical Tests}

Since there were no constraints in place, computation where carried out via an implementation of a spectral gradient type method 
\cite{BirginMartinezRaydan2000} under Python.
We first tested our approach with the widely known colon-cancer data set\footnote{\url{https://www.csie.ntu.edu.tw/~cjlin/libsvmtools/datasets/binary.html}} with 2000 features and 62 data samples. We initiated a training set by choosing 42 data samples, consisting of 14 positive and 28 negative labels. As starting parameters we chose the penalty function $p_i^\rho(y_i) = 1/2 (y_i - \sqrt{2\rho})^2$ and set $(x^0, y^0) = (0, \sqrt{2\rho} e)$ as the initial guess, where $\rho = 0.1 \cdot \frac{1}{m}$. Furthermore, we set $\alpha_0 = 0.1, \ \beta = 10$. The spectral gradient step was executed for $10^{4}$ iterations or to a precision of $10^{-5}$, where we accepted the end result once complementarity between $x^k$ and $y^k$ was reached to a precision of $10^{-6}$. The accuracy measured to around 75\% as integral of the ROC-curve, where $\norm{x}_0 = 7$ from $2000$ possible entries, whereas cpu time accrued to 2.35 seconds. In fact, our solution vector predicted $0$ and $1$ label to machine precision so that no matter a given threshold we would always correctly guess exactly $15$ out of $20$ possible instances in the validation set.

Second we chose as test example the gisette data set from the NIPS 2003 challenge\footnote{\url{https://archive.ics.uci.edu/}}. The gisette data sets comes with specific training and validation data. Again, the initial guess was made with $(x^0,y^0) = (0,\sqrt{2\rho} e)$, where we used the same penalty function as before and set $\rho = 1/m$. The remaining parameters where chosen as $\alpha_0 = 1$, $\beta = 10$. The spectral gradient step was executed for $10^4$ iterations or to a precision of $10^{-2}$, where we accepted $10^{-2}$ as tolerance for the complementarity constraints. The accuracy measured to around 99.45 \% as integral of the ROC-curve, where $\norm{x}_0 = 551$ of $5000$ possible entries, whereas cpu time accrued to 2.75 minutes. If we accept $0.5$ as the threshold to which we predict a $1$ or $0$ if a value larger or less than $0.5$ was observed, we would correctly guess in 97.2\% of instances in the validation set.  
\begin{figure}
    \centering
    \includegraphics[width = .5\textwidth, height = 6cm]{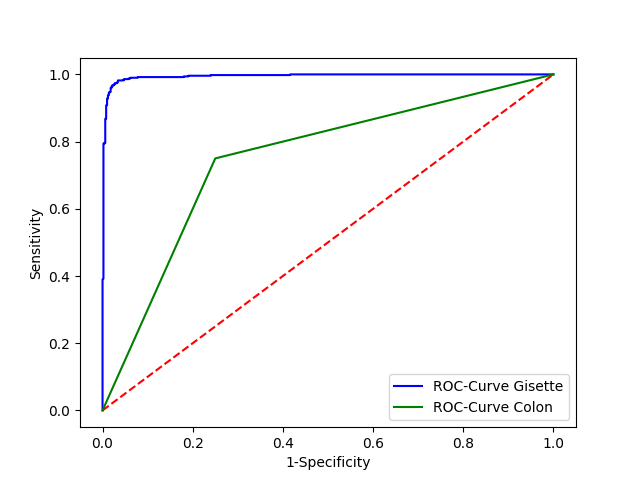}
    \caption{ROC-Curve for the gisette and the colon data set.}
    \label{fig:ROCCurve}
\end{figure}

\subsection{Support Vector Machines}

\subsubsection{Problem Definition}

The support vector machine problem may be stated as
$$ 
   \min_{c,\gamma} \frac{1}{2m} \norm{c}^2_2 + \rho \norm{\max\{0, e - z \circ (A c - \gamma)\}}_0
$$
with some matrix $A \in \RR^{m \times n}$, and a $z \in \{-1,1\}^m$ signaling that the $i$-th sample $a_i^T$ given as $i$-th row of $A$ belongs to the class $z_i$. By introducing an auxiliary variable $u$, we may rewrite this problem into
\begin{equation} \min_{c,\gamma,u} \frac{1}{2m} \norm{c}^2_2 + \rho \norm{u}_0, \quad u \ge 0, \ u \ge e-z\circ(Ac-\gamma). \end{equation}
The corresponding penalized problem is given by
\begin{equation}
    \min_{c,\gamma,u} \frac{1}{2m} \norm{c}^2_2 + p^{\rho}(y) + \alpha y^T u, \quad \text{s.t.} \quad u \ge 0,\ u \ge e - z \circ (Ac - \gamma).
\end{equation}
These subproblems are then solved by an augmented Lagrangian 
approach.

\subsubsection{Numerical Tests}

We applied our method to a few selected datasets from the source\footnote{\url{https://www.csie.ntu.edu.tw/~cjlin/libsvmtools/datasets/binary.html}}. The choice of the penalty function was again $p_i^\rho(y_i) = 1/2 (y_i - \sqrt{2 \rho})^2$. The additional scaling by factor $1/m$ in the target function was introduced to avoid large values during the gradient method which occured with large sample-size. We met the matching choice of $\rho = 1/m, \ \alpha_0 = 1/m$ and set $\beta = 10$. The value of $\delta$ in \ref{PenaltyAlg} was reduced to $10^{-2}$. We compared our results (denoted by 'pnl') to the libsvm solver available as python module\footnote{\url{https://pypi.org/project/libsvm/}}. The table \ref{table:SVM} suggests that our penalty approach is of particular interest once the size of features and training variables becomes exceedingly large.
\begin{table}
\begin{tabular}{ |c|c|c|c|c|c|c|c| } 
 \hline
 Dataset & features & train & test & acc-pnl & acc-libsvm & cpu-pnl & cpu-libsvm \\
 \hline
 \texttt{arcene} & 10000 & 100 & 100 &82\% & 83\% & 10.5 sec & 0.4 sec\\
 \texttt{jcnn} & 22 & 49990 & 91701 & 91.8 \% & 92.1 \% & 105.7 sec & 11.1 sec\\
 \texttt{a9a} & 123 & 23373 & 8141 & 84.9\% & 85\% & 145sec & 15 sec \\
 \texttt{binary} & 47236 & 20242 & 677399 & 96.3\% & 96.3 \% & 14.7 sec & 90 sec \\
 \texttt{kddb} & 1129522 & 19264097 & 748401 & 94.4\% & - & 161.8 sec & -\\
 \hline
\end{tabular}
 \caption{ \label{table:SVM}
Accuracy and CPU-time for the SVM tests.}
\end{table}
In the last case, the call to the \texttt{svm\_train} function within the libsvm package did not yield any result.

\subsection{Dictionary Learning}

\subsubsection{Problem Definition}

The dictionary learning problem can be understood as an extension to the basis pursuit denoising type of problem, where also the basis is searched for. Let $Z \in \RR^{n \times m}$ be given. We look for $ D\in \mathbb{R}^{l \times n}, \ C \in \mathbb{R}^{l \times m}$ which minimize
$$ 
   \min_{D,C} \ \frac{1}{2} \norm{Z - D^T C}_F^2 + \rho \norm{C}_0, \quad \text{s.t.} \quad \norm{D^T_j}_2^2 \le 1 \quad \forall j = {1,...,l},
$$
where $D^T_j$ denote the rows in $D$, $\norm{\cdot}_F$ is the Frobenius norm and we define
$$ \norm{C}_0 := \sum_{i,j = 1}^{n,m} \norm{C_{i,j}}_0.$$ As with the logistic regression example seen before, there are no sign constraints with $C$. We again pass to the shift
$$ C = C_+ - C_-, \quad C_+ \ge0, \ C_- \ge0.$$
Note that there are no constraints with respect to $C$, the variable we want to be sparse. Furthermore, the feasible set $X$ is, in particular, closed and convex and as such has a unique projection, which is, in this case, also easy to compute. 
Let $$ F(C_+,C_-,D) = \frac{1}{2} \norm{Z - D^T (C_+ - C_-)}_F^2.$$
We therefore require in step 2 of algorithm \ref{PenaltyAlg}
$$ \norm{P_X(D^{k+1} - \nabla_D F(C_+^{k+1},C_-^{k+1},D^{k+1})) - D^{k+1}} \le \varepsilon_k.$$
By passing to the limit every accumulation point is already stationary with respect to component $D$. It seems natural to simply apply a projected spectral gradient method to this type of problem. Notice 
that, in particular, the derivative with respect to $C$ and $D$ is given by
$$ \nabla_{C_{\pm}} F(C_+,C_-,D) = \pm \left(DD^T C - DZ\right), \quad \nabla_D F(C_+,C_-,D) = CC^T D - CZ^T.$$
For the choice $(C_+,C_-,D) = (0,0,0)$ the above expressions both vanish and we clearly have an S-stationary point. We therefore initalized $(C_+^0,C_-^0,D^0)$ as a random matrix $\RR^{l \times 2m + n}$ with entries taken from a standard normal distribution, projected onto $X$.

\subsubsection{Numerical Tests}

We conducted our numerical tests similar as in \cite{DeMarchi2023}. In 100 instances with $n = 10,\ l=20, \ m=30$ we generated $Z = C^T D$ from primary matrices $C$ and $D$, where $C$ only had three nonzero entries at random positions per column, where the values where taken from a standard normal distribution and where $D$ was chosen as a random standard normal matrix with normalized rows. We again chose the penalty function $p_i^\rho(y_i) = 1/2 (y_i - \sqrt{2\rho})^2$ and set $\rho = 0.1$, $\alpha_0 = 0.1$, $\beta = 10$. The spectral gradient
method was run for $10^{4}$ iterations with a tolerance of $10^{-3}$. Complementarity between $C_\pm$ and $Y_\pm$ was accepted to a tolerance of $10^{-6}$. We detail the achieved values as well as the required cpu time in the chart (\ref{fig:LearningDict}). We measure the results against a proximal gradient type method applied to the same test set (\ref{fig:LearningDictProx}) (note the different scaling of the axes
regarding the objective function values). The proximal gradient method was mirrored from aformentioned source \cite{DeMarchi2023} using the averaging stepsize as detailed there. While computation time with the proximal gradient is lower compared to our exact penalty approach, we are able to significantly improve upon the reached target value. 
\begin{figure}[ht]
    \begin{subfigure}{.49\textwidth}
    \includegraphics[width = .8\linewidth]{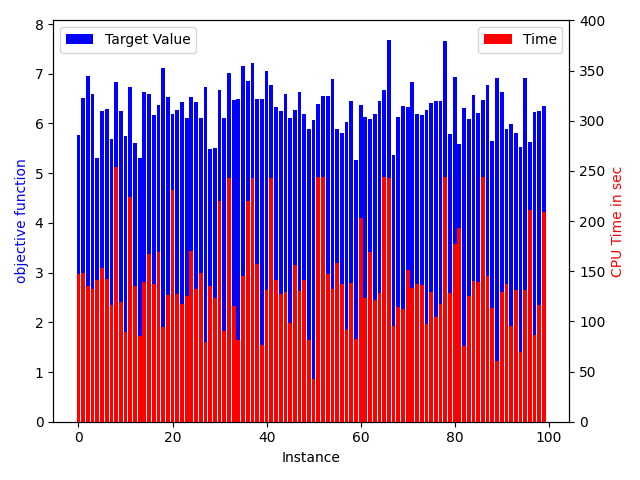}
    \caption{Computation time and achieved target \\ value via the exact penalty approach.}
    \label{fig:LearningDict}
    \end{subfigure}
    \begin{subfigure}{.49\textwidth}
    \includegraphics[width = .8\linewidth]{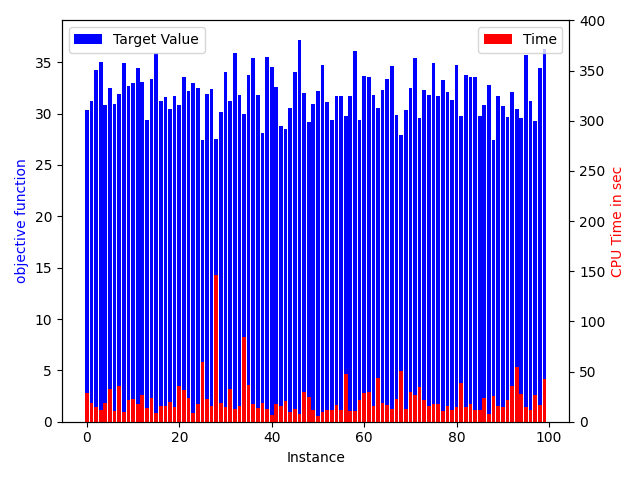}
    \caption{Computation time and achieved target \\ value via a proximal gradient type method.}
    \label{fig:LearningDictProx}
    \end{subfigure}
    \caption{A comparison of the penalty and proixmal gradient method applied to the Dictionary Learning problem}
\end{figure}

The second test regarding dictionary-learning-type problems was done with the MNIST\footnote{Y. LeCun and C. Cortes. Mnist handwritten digit database. AT\&T Labs [Online]. Available: \url{https://yann.lecun.com/exdb/mnist}, 2010} data set. We have chosen the first 100 images and used and tried to find a sparse representation $Y_i \approx C^T_i D_i$ for each of the images $Y_i$. As $Y_i$ was represented by a 28 by 28 matrix, we let $C_i,\ D_i \in \RR^{l \times 28}$ for $l = 6,8,10,12,14$ and survey the interesting characteristics in the Table~\ref{tab:DictL} as an average over the $100$ test runs. To get an idea for the quality of the achieved decomposition, we compare for $i = 1,...,18$ the original image $Y_i$ to the result $C_i^T D_i$ specifically for dimension $l = 10$ in figure (\ref{fig:MINTDecomp}). 
\begin{center}
\begin{figure}[ht]
\begin{subfigure}[h]{.49\textwidth}
\begin{tabular}{ |c|c|c|c| }
 \hline
 dim & $f + .1 \norm{\cdot}_0$ & $\norm{\cdot}_0$ & time\\
 \hline
 $l = 6$ & 1.2 & 23.3 & 16.7 \\
 $l = 8$ & 0.6 & 21.8 & 16.3\\
 $l = 10$ & 0.3 & 19.9 & 14.3\\
 $l = 12$ & 0.2 & 17.6 & 12.1\\
 $l = 14$ & 0.09 & 16.3 & 10\\
 \hline
\end{tabular}
\caption{Comparison of the MNIST Dataset \\ for $100$ instances.}
\label{tab:DictL}
\end{subfigure}
\begin{subfigure}[h]{.49\textwidth}
        \includegraphics[width = \textwidth, height = 6cm]{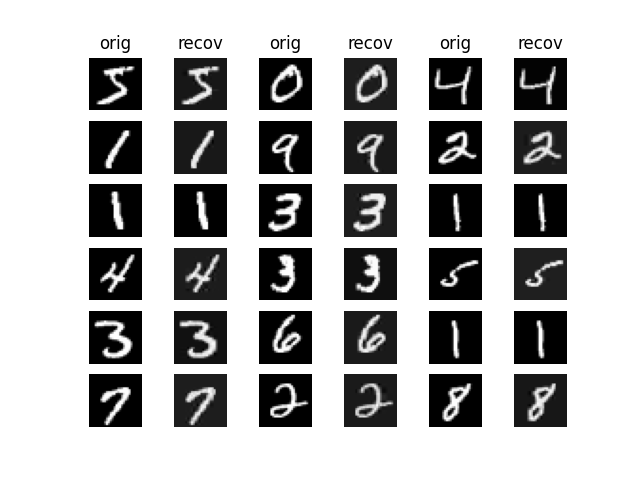}
    \caption{Comparison of the original images (odd columns) \\ to the recovered images (even columns)}
    \label{fig:MINTDecomp}
\end{subfigure}
\caption{The penalty method applied to MNIST in order to find sparse representation of images.}
\end{figure}
\end{center}

\section{Final Remarks}\label{Sec:Final}

This paper introduces a class of reformulations of the 
$ \ell_0 $-sparse optimization problem and develops suitable
constraint qualifications as well as corresponding first-
and second-order optimality conditions. The results are then
used to apply an exact penalty-type method for the solution 
of the $ \ell_0 $-sparse optimization problem which is particularly
useful if the contraints include nonnegativity conditions on the 
variables. Otherwise, one has to use a split of the free
variables which might introduce additional local minima.
Though the corresponding numerical results are still very
promising, in this situation, it might be more favourable to
apply another technique based on our reformulation which can
be applied also in the case where there exist free variables.
One natural possibility is the augmented Lagrangian method, 
and a closer look at this technique will therefore be part of our future
research.

\bibliographystyle{habbrv}
\bibliography{references_penalty}

\end{document}